\documentclass[a4paper,11pt,fleqn]{article}
\usepackage[obeyspaces,hyphens,spaces]{url}
\usepackage[dvipsnames]{xcolor}
\usepackage{amsmath}
\usepackage{amssymb}
\usepackage{amsfonts}
\usepackage{mathtools}
\usepackage{euscript}
\usepackage{pifont}          
\usepackage{theorem}
\usepackage{charter}
\usepackage{srcltx}          
\usepackage{etoolbox}
\usepackage[scr=rsfs]{mathalfa}
\allowdisplaybreaks 
\definecolor{labelkey}{HTML}{0455BF}
\definecolor{refkey}{rgb}{0,0.6,0.0}
\definecolor{dblue}{HTML}{044EAF}
\definecolor{dgreen}{HTML}{02724A}
\definecolor{myellow}{HTML}{D97904}
\definecolor{dred}{HTML}{D90404}
\usepackage{upref}
\usepackage{hyperref}
\hypersetup{colorlinks=true,%
linktocpage=true,%
linkcolor=dblue,%
citecolor=dgreen,%
urlcolor=dred}

\renewcommand\familydefault{\rmdefault}
\DeclareMathAlphabet{\mathrm}{OT1}{\familydefault}{m}{n}
\makeatletter
\renewcommand\operator@font{\rm}
\makeatother

\newcommand{\scal}[2]{{\langle{{#1}\mid{#2}}\rangle}}

\newcommand{\pair}[2]{\langle{{#1},{#2}}\rangle} 
 
\newcommand{\Pair}[2]{\big\langle{{#1},{#2}}\big\rangle}

\newcommand{\menge}[2]{\big\{{#1}~|~{#2}\big\}} 
\newcommand{\Menge}[2]{\left\{{#1}~\middle|~{#2}\right\}} 
\newcommand{\EE}{\ensuremath{\mathsf{E}}}
\newcommand{\PP}{\ensuremath{\mathsf{P}}}

\newcommand{\XX}{\ensuremath{\mathcal X}}
\newcommand{\TXX}{\ensuremath{\widetilde{\mathcal X}}}
\newcommand{\YY}{\ensuremath{\mathcal Y}}
\newcommand{\TYY}{\ensuremath{\widetilde{\mathcal Y}}}

\newcommand{\RR}{\ensuremath{\mathbb R}}
\newcommand{\NN}{\ensuremath{\mathbb N}}

\newcommand{\XS}{\ensuremath{\mathsf X}}
\newcommand{\YS}{\ensuremath{\mathsf Y}}
\newcommand{\ZS}{\ensuremath{\mathsf Z}}
\newcommand{\BE}{\ensuremath{\EuScript B}}
\newcommand{\FF}{\ensuremath{\EuScript F}}

\renewcommand{\leq}{\ensuremath{\leqslant}}
\renewcommand{\geq}{\ensuremath{\geqslant}}

\newcommand{\pinf}{\ensuremath{{{+}\infty}}}
\newcommand{\minf}{\ensuremath{{{-}\infty}}}

\newcommand{\RX}{\ensuremath{\left]{-}\infty,{+}\infty\right]}}
\newcommand{\RXX}{\overline{\mathbb{R}}}

\newcommand{\RPP}{\ensuremath{\left]0,{+}\infty\right[}}

\newcommand{\emp}{\ensuremath{\varnothing}}
\newcommand{\Id}{\ensuremath{\mathrm{Id}}}

\newcommand{\exi}{\ensuremath{\exists\,}}

\newcommand{\mae}{\ensuremath{\text{\rm$\mu$-a.e.}}}
\newcommand{\moyo}[2]{\prescript{#2}{}{#1}}
\DeclareMathOperator*{\essinf}{ess\,inf}

\DeclareMathOperator{\intdom}{int\,dom}

\DeclareMathOperator{\epi}{epi}

\DeclareMathOperator{\dom}{dom}

\DeclareMathOperator{\inte}{int}

\DeclareMathOperator{\rec}{rec}
\DeclareMathOperator{\prox}{prox}

\newtheorem{theorem}{Theorem}[section]
\newtheorem{lemma}[theorem]{Lemma}

\newtheorem{proposition}[theorem]{Proposition}

\theoremstyle{plain}{\theorembodyfont{\rmfamily}%
\newtheorem{assumption}[theorem]{Assumption}}
\theoremstyle{plain}{\theorembodyfont{\rmfamily}%
\newtheorem{example}[theorem]{Example}}
\theoremstyle{plain}{\theorembodyfont{\rmfamily}%
\newtheorem{remark}[theorem]{Remark}}
\theoremstyle{plain}{\theorembodyfont{\rmfamily}%
}
\theoremstyle{plain}{\theorembodyfont{\rmfamily}%
}
\theoremstyle{plain}{\theorembodyfont{\rmfamily}%
}
\theoremstyle{plain}{\theorembodyfont{\rmfamily}%
\newtheorem{definition}[theorem]{Definition}}
\theoremstyle{plain}{\theorembodyfont{\rmfamily}%
}
\theoremstyle{plain}{\theorembodyfont{\rmfamily}%
\newtheorem{notation}[theorem]{Notation}}
\usepackage{enumitem}

\numberwithin{equation}{section}
\setlist[enumerate]{itemsep=-1pt,topsep=1pt}
\setlist[description]{itemsep=-1pt,topsep=1pt}
\setlist[itemize]{itemsep=-1pt,topsep=1pt}
\usepackage{soul}
\setstcolor{dred}
\newcommand*\Cdot{{\mkern 2mu\cdot\mkern 2mu}}

\usepackage{authblk}
\newcommand{\email}[1]{\href{mailto:#1}{\nolinkurl{#1}}}

\begin{document}

\title{\sffamily\huge%
Interchange Rules for Integral Functions\thanks{%
Contact author: P. L. Combettes.
Email: \email{plc@math.ncsu.edu}.
Phone: +1 919 515 2671.
The work of M. N. B\`ui was supported by NAWI Graz and
the work of P. L. Combettes was supported by the National
Science Foundation under grant DMS-1818946.
}}

\author[1]{Minh N. B\`ui}
\affil[1]{Universit\"at Graz
\authorcr
Institut f\"ur Mathematik und Wissenschaftliches Rechnen
\authorcr
8010 Graz, Austria
\authorcr
\email{minh.bui@uni-graz.at}\medskip
}
\author[2]{Patrick L. Combettes}
\affil[2]{North Carolina State University
\authorcr
Department of Mathematics
\authorcr
Raleigh, NC 27695-8205, USA
\authorcr
\email{plc@math.ncsu.edu}
}

\date{~}

\maketitle

\begin{abstract} 
We first present an abstract principle for the interchange of
infimization and integration over spaces of mappings taking values
in topological spaces. New conditions on the underlying space and
the integrand are then introduced to convert this principle into
concrete scenarios that are shown to capture those of various
existing interchange rules. These results are leveraged to
improve state-of-the-art interchange rules for evaluating Legendre
conjugates, subdifferentials, recessions, Moreau envelopes, and
proximity operators of integral functions by bringing the
corresponding operations under the integral sign. 
\end{abstract}


\begin{MSC}
46G12, 49J52, 46N10
\end{MSC}

\newpage

\section{Introduction}
\label{sec:1}

This paper concerns the interchange of the infimization and
integration operations in the context of the following assumption.

\begin{assumption}
\label{a:1}
\
\begin{enumerate}[label={\rm[\Alph*]}]
\item
\label{a:1a}
$\XS$ is a real vector space endowed with a
Souslin topology $\EuScript{T}_\XS$ and
associated Borel $\sigma$-algebra $\BE_\XS$.
\item
\label{a:1b}
The mapping
$(\XS\times\XS,\BE_\XS\otimes\BE_\XS)\to(\XS,\BE_\XS)\colon
(\mathsf{x},\mathsf{y})\mapsto\mathsf{x}+\mathsf{y}$
is measurable.
\item
\label{a:1c}
For every $\lambda\in\RR$, the mapping
$(\XS,\BE_\XS)\to(\XS,\BE_\XS)\colon
\mathsf{x}\mapsto\lambda\mathsf{x}$ is measurable.
\item
\label{a:1d}
$(\Omega,\FF,\mu)$ is a $\sigma$-finite measure space such that
$\mu(\Omega)\neq 0$, and $\mathcal{L}(\Omega;\XS)$
denotes the vector space of measurable mappings from
$(\Omega,\FF)$ to $(\XS,\BE_\XS)$.
\item
\label{a:1e}
$\XX$ is a vector subspace of $\mathcal{L}(\Omega;\XS)$.
\item
\label{a:1f}
$\varphi\colon(\Omega\times\XS,\FF\otimes\BE_\XS)\to\RXX$ is an
integrand in the sense that it is measurable and, for every
$\omega\in\Omega$, $\epi\varphi_\omega\neq\emp$, where
$\varphi_\omega=\varphi(\omega,\Cdot)$.
\item
\label{a:1g}
There exists $\overline{x}\in\XX$ such that
$\int_\Omega\max\{\varphi(\Cdot,\overline{x}(\Cdot)),0\}
d\mu<\pinf$.
\end{enumerate}
As is customary, given a measurable function
$\varrho\colon(\Omega,\FF)\to\RXX$, $\int_\Omega\varrho d\mu$ is
the usual Lebesgue integral, except when the Lebesgue integral
$\int_\Omega\max\{\varrho,0\}d\mu$ is $\pinf$, in which case
$\int_\Omega\varrho d\mu=\pinf$.
\end{assumption}

Many problems in analysis and its applications require the
evaluation of the infimum over $\XX$ of the function
$f\colon x\mapsto\int_\Omega\varphi(\Cdot,x(\Cdot))d\mu$. 
A simpler task is to evaluate the function 
$\phi\colon\omega\mapsto\inf\varphi(\omega,\XS)$ and then compute 
$\int_\Omega\phi d\mu$. In general, this provides only a
lower bound as $\inf f(\XX)\geq\int_\Omega\phi d\mu$. Conditions
under which the two quantities are equal have been established in
\cite{Hiai77}, \cite{Perk18}, and \cite{Roc76k} under various
hypotheses on $\XS$, $(\Omega,\FF,\mu)$, $\XX$, and $\varphi$.
The resulting infimization-integration interchange rule
is a central tool in areas such as 
plasticity theory \cite{Bouc88},
convex analysis \cite{Corr19},
multivariate analysis \cite{Hiai77},
calculus of variations \cite{Ioff74},
economics \cite{Levi85},
stochastic processes \cite{Penn18},
optimal transport \cite{Penn19},
stochastic optimization \cite{Penn23},
finance \cite{Perk18},
variational analysis \cite{Rock09},
and stochastic programming \cite{Shap21}.
Note that, in Assumption~\ref{a:1}\ref{a:1a}--\ref{a:1c}, we do
not require that $(\XS,\EuScript{T}_\XS)$ be a topological vector
space to accommodate certain applications. For instance, in
\cite{Perk18}, $\XS$ is the space of c\`adl\`ag functions on
$[0,1]$ and $\EuScript{T}_\XS$ is the Skorokhod topology. In this
context, $(\XS,\EuScript{T}_\XS)$ is a Polish space
\cite[Chapter~3]{Bill68} which is not a topological vector space
\cite{Pest95} but which satisfies
Assumption~\ref{a:1}\ref{a:1a}--\ref{a:1c}.

Our first contribution is Theorem~\ref{t:1} below, which provides,
under the umbrella of Assumption~\ref{a:1}, a broad setting for
the interchange of infimization and integration.

\begin{theorem}[interchange principle]
\label{t:1}
Suppose that Assumption~\ref{a:1} and the following hold:
\begin{enumerate}
\item
\label{t:1i}
$\inf_{\mathsf{x}\in\XS}\varphi(\Cdot,\mathsf{x})$ is
$\FF$-measurable.
\item
\label{t:1ii}
There exists a sequence $(x_n)_{n\in\NN}$ in
$\mathcal{L}(\Omega;\XS)$ such that the following are
satisfied:
\begin{enumerate}
\item
\label{t:1iia}
$\inf_{\mathsf{x}\in\XS}\varphi(\Cdot,\mathsf{x})=
\inf_{n\in\NN}\varphi(\Cdot,x_n(\Cdot)+\overline{x}(\Cdot))$
$\mae$
\item
\label{t:1iib}
There exists an increasing sequence $(\Omega_k)_{k\in\NN}$
of finite $\mu$-measure sets in $\FF$ such that
$\bigcup_{k\in\NN}\Omega_k=\Omega$ and
\begin{equation}
\label{e:99}
(\forall n\in\NN)(\forall k\in\NN)\quad
\menge{1_Ax_n}{\FF\ni A\subset\Omega_k\,\,\text{and}\,\,
\overline{x_n(A)}\,\,\text{is compact}}\subset\XX.
\end{equation}
\end{enumerate}
\end{enumerate}
Then
\begin{equation}
\label{e:1}
\inf_{x\in\XX}\int_\Omega\varphi\big(\omega,x(\omega)\big)
\mu(d\omega)=
\int_\Omega\inf_{\mathsf{x}\in\XS}\varphi(\omega,\mathsf{x})\,
\mu(d\omega).
\end{equation}
\end{theorem}

Theorem~\ref{t:1} is proved in Section~\ref{sec:3}. The second
contribution is the introduction of two new tools --- compliant
spaces and an extended notion of normal integrands. This is done in
Section~\ref{sec:4}, where these notions are illustrated through
various examples. In Section~\ref{sec:5}, compliance and normality
are utilized to build a pathway between the abstract interchange
principle of Theorem~\ref{t:1} and separate conditions
on $\XX$ and $\varphi$ that capture various application settings.
The main result of that section is Theorem~\ref{t:8}, which
encompasses in particular the interchange rules of
\cite{Hiai77,Perk18,Roc76k}, as well as those implicitly present
in \cite{Roc68a,Rock71,Vala75}. These different frameworks have so
far not been brought together and we improve them in several
directions, for instance by not requiring the completeness of
$(\Omega,\FF,\mu)$ and by relaxing the assumptions on $\XS$. This
leads to new concrete scenarios under which \eqref{e:1} holds.
Our third contribution, presented in Section~\ref{sec:6}, concerns
convex-analytical operations on integral functions. By combining
Theorem~\ref{t:1}, compliance, and normality, we broaden
conditions for evaluating Legendre conjugates, subdifferentials,
recessions, Moreau envelopes, and proximity operators of integral
functions by bringing the corresponding operations under the
integral sign. These results improve state-of-the-art convex
calculus rules from
\cite{Livre1,Penn18,Penn23,Rock71,Roc76k,Vala75}.

\section{Notation and background}
\label{sec:2}

\subsection{Measure theory}

We set $\RXX=\left[{-}\infty,{+}\infty\right]$. Let $(\Omega,\FF)$
be a measurable space and let $A$ be a subset of $\Omega$. The
characteristic function of $A$ is denoted by $1_A$ and the
complement of $A$ is denoted by $\complement A$. Now let
$(\XS,\EuScript{T}_\XS)$ be a Hausdorff topological space with
Borel $\sigma$-algebra $\BE_\XS$. We denote by
$\mathcal{L}(\Omega;\XS)$ the vector space of measurable mappings
from $(\Omega,\FF)$ to $(\XS,\BE_\XS)$. Given a measure $\mu$ on
$(\Omega,\FF)$, $\mathcal{L}^1(\Omega;\RR)$ is the subset of
$\mathcal{L}(\Omega;\RR)$ of integrable functions, and
$\mathcal{L}^1(\Omega;\RXX)$ is defined likewise. Given a separable
Banach space $(\XS,\|\Cdot\|_\XS)$, we set
$\mathcal{L}^\infty(\Omega;\XS)=
\menge{x\in\mathcal{L}(\Omega;\XS)}{\sup\|x(\Omega)\|_\XS<\pinf}$.

\subsection{Topological spaces}
Given topological spaces $(\YS,\EuScript{T}_\YS)$ and
$(\ZS,\EuScript{T}_\ZS)$,
$\EuScript{T}_\YS\boxtimes\EuScript{T}_\ZS$ denotes the
standard product topology.

Let $(\XS,\EuScript{T}_\XS)$ be a Hausdorff topological space.
The Borel $\sigma$-algebra of $(\XS,\EuScript{T}_\XS)$
is denoted by $\BE_\XS$. Furthermore, $(\XS,\EuScript{T}_\XS)$ is:
\begin{itemize}
\item
regular \cite[Section~I.8.4]{Bour71} if, for every closed
subset $\mathsf{C}$ of $(\XS,\EuScript{T}_\XS)$
and every $\mathsf{x}\in\complement\mathsf{C}$, there exist
$\mathsf{V}\in\EuScript{T}_\XS$ and
$\mathsf{W}\in\EuScript{T}_\XS$ such that
$\mathsf{C}\subset\mathsf{V}$, $\mathsf{x}\in\mathsf{W}$, and
$\mathsf{V}\cap\mathsf{W}=\emp$;
\item
a Polish space \cite[Section~IX.6.1]{Bour74} if it is separable and
there exists a distance $\mathsf{d}$ on $\XS$ that induces the same
topology as $\EuScript{T}_\XS$ and such that $(\XS,\mathsf{d})$ is
a complete metric space;
\item
a Souslin space \cite[Section~IX.6.2]{Bour74} if there exist a
Polish space $(\YS,\EuScript{T}_\YS)$ and a continuous surjective
mapping from $(\YS,\EuScript{T}_\YS)$ to $(\XS,\EuScript{T}_\XS)$;
\item
a Lusin space \cite[Section~IX.6.4]{Bour74} if there exists a
topology $\widetilde{\EuScript{T}_\XS}$ on $\XS$ such that
$\EuScript{T}_\XS\subset\widetilde{\EuScript{T}_\XS}$ and
$(\XS,\widetilde{\EuScript{T}_\XS})$ is a Polish space;
\item
a Fr\'echet space \cite[Section~II.4.1]{Bour81} if it is a locally
convex real topological vector space and there exists a
distance $\mathsf{d}$ on $\XS$ that induces
the same topology as $\EuScript{T}_\XS$ and such that
$(\XS,\mathsf{d})$ is a complete metric space.
\end{itemize}
Now let $\mathsf{f}\colon\XS\to\RXX$.
The epigraph of $\mathsf{f}$ is
\begin{equation}
\epi\mathsf{f}=\menge{(\mathsf{x},\xi)\in\XS\times\RR}{
\mathsf{f}(\mathsf{x})\leq\xi},
\end{equation}
$\mathsf{f}$ is proper if
$\minf\notin\mathsf{f}(\XS)\neq\{\pinf\}$, and $\mathsf{f}$ is 
$\EuScript{T}_\XS$-lower semicontinuous if
$\epi\mathsf{f}$ is
$\EuScript{T}_\XS\boxtimes\EuScript{T}_\RR$-closed.

\subsection{Duality}

The dual of a real topological vector space
$(\XS,\EuScript{T}_\XS)$, that is, the vector space of continuous
linear functionals on $(\XS,\EuScript{T}_\XS)$, is denoted by
$(\XS,\EuScript{T}_\XS)^*$.

Let $\XS$ and $\YS$ be real vector spaces which are in
separating duality via a bilinear form
$\pair{\Cdot}{\Cdot}_{\XS,\YS}\colon\XS\times\YS\to\RR$, that is
\cite[Section~II.6.1]{Bour81},
\begin{equation}
\begin{cases}
(\forall\mathsf{x}\in\XS)\quad
\pair{\mathsf{x}}{\Cdot}_{\XS,\YS}=0
\quad\Rightarrow\quad
\mathsf{x}=\mathsf{0}
\\
(\forall\mathsf{y}\in\YS)\quad
\pair{\Cdot}{\mathsf{y}}_{\XS,\YS}=0
\quad\Rightarrow\quad
\mathsf{y}=\mathsf{0}.
\end{cases}
\end{equation}
In addition, equip $\XS$ with a locally convex topology
$\EuScript{T}_\XS$ which is compatible with the pairing
$\pair{\Cdot}{\Cdot}_{\XS,\YS}$ in the sense that 
$(\XS,\EuScript{T}_\XS)^*
=\{\pair{\Cdot}{\mathsf{y}}_{\XS,\YS}\}_{\mathsf{y}\in\YS}$
and, likewise, equip $\YS$ with a locally convex topology
$\EuScript{T}_\YS$ which is compatible with the pairing
$\pair{\Cdot}{\Cdot}_{\XS,\YS}$ in the sense that
$(\YS,\EuScript{T}_\YS)^*
=\{\pair{\mathsf{x}}{\Cdot}_{\XS,\YS}\}_{\mathsf{x}\in\XS}$
\cite[Section~IV.1.1]{Bour81}.
Following \cite{More66}, the Legendre conjugate of 
$\mathsf{f}\colon\XS\to\RXX$ is
\begin{equation}
\label{e:l0d}
\mathsf{f}^*\colon\YS\to\RXX\colon
\mathsf{y}\mapsto\sup_{\mathsf{x}\in\XS}
\big(\pair{\mathsf{x}}{\mathsf{y}}_{\XS,\YS}-
\mathsf{f}(\mathsf{x})\big)
\end{equation}
and the Legendre conjugate of $\mathsf{g}\colon\YS\to\RXX$ is
\begin{equation}
\mathsf{g}^*\colon\XS\to\RXX\colon
\mathsf{x}\mapsto\sup_{\mathsf{y}\in\YS}
\big(\pair{\mathsf{x}}{\mathsf{y}}_{\XS,\YS}-
\mathsf{g}(\mathsf{y})\big).
\end{equation}
Let $\mathsf{f}\colon\XS\to\RXX$. If $\mathsf{f}$ is proper, its
subdifferential is the set-valued operator
\begin{equation}
\label{e:s14}
\begin{aligned}
\partial\mathsf{f}\colon\XS&\to 2^\YS\\
\mathsf{x}&\mapsto\menge{\mathsf{y}\in\YS}{
(\forall\mathsf{z}\in\XS)\,\,
\pair{\mathsf{z}-\mathsf{x}}{\mathsf{y}}_{\XS,\YS}
+\mathsf{f}(\mathsf{x})\leq\mathsf{f}(\mathsf{z})}
=\menge{\mathsf{y}\in\YS}{
\mathsf{f}(\mathsf{x})+\mathsf{f}^*(\mathsf{y})=
\pair{\mathsf{x}}{\mathsf{y}}_{\XS,\YS}}.
\end{aligned}
\end{equation}
In addition, $\mathsf{f}$ is convex if $\epi\mathsf{f}$ is a convex
subset of $\XS\times\RR$, and $\Gamma_0(\XS)$ denotes the class of
proper lower semicontinuous convex functions from $\XS$ to $\RX$.
Suppose that $\mathsf{f}\in\Gamma_0(\XS)$ and let 
$\mathsf{z}\in\dom\mathsf{f}$. The recession function of
$\mathsf{f}$ is the function in $\Gamma_0(\XS)$ defined by
\begin{equation}
\label{e:r}
\rec\mathsf{f}\colon\XS\to\RX\colon\mathsf{x}\mapsto
\lim_{0<\alpha\uparrow\pinf}
\frac{\mathsf{f}(\mathsf{z}+\alpha\mathsf{x})
-\mathsf{f}(\mathsf{z})}{\alpha}.
\end{equation}
Now suppose that, in addition, $\XS=\YS$ is Hilbertian
and $\pair{\Cdot}{\Cdot}_{\XS,\YS}$ is the scalar product of
$\XS$, and let $\gamma\in\RPP$. The Moreau envelope of
$\mathsf{f}$ of index $\gamma$ is the function in $\Gamma_0(\XS)$
defined by
\begin{equation}
\label{e:7}
\moyo{\mathsf{f}}{\gamma}\colon\XS\to\RR\colon
\mathsf{x}\mapsto
\min_{\mathsf{y}\in\XS}\bigg(\mathsf{f}(\mathsf{y})
+\dfrac{1}{2\gamma}\|\mathsf{x}-\mathsf{y}\|_\XS^2\bigg)
\end{equation}
and the proximal point of $\mathsf{x}\in\XS$ relative to
$\gamma\mathsf{f}$ is the unique point
$\prox_{\gamma\mathsf{f}}\mathsf{x}\in\XS$ such that
\begin{equation}
\label{e:7b}
\moyo{\mathsf{f}}{\gamma}(\mathsf{x})
=\mathsf{f}(\prox_{\gamma\mathsf{f}}\mathsf{x})
+\dfrac{1}{2\gamma}
\|\mathsf{x}-\prox_{\gamma\mathsf{f}}\mathsf{x}\|_\XS^2.
\end{equation}
The proximity operator $\prox_{\gamma\mathsf{f}}\colon\XS\to\XS$
thus defined can be expressed as
\begin{equation}
\label{e:8}
\prox_{\gamma\mathsf{f}}=(\Id+\gamma\partial\mathsf{f})^{-1}.
\end{equation}

\section{Proof of the interchange principle}
\label{sec:3}

Proving Theorem~\ref{t:1} necessitates a few technical facts.

\begin{lemma}
\label{l:8}
Let $(\Omega,\FF)$ be a measurable space,
let $n$ be a strictly positive integer,
and let $(\varrho_i)_{0\leq i\leq n}$ be a family in
$\mathcal{L}(\Omega;\RR)$. Then there exists a family
$(B_i)_{0\leq i\leq n}$ in $\FF$ such that 
\begin{equation}
\label{e:bx}
(B_i)_{0\leq i\leq n}\,\,\text{are pairwise disjoint},
\quad\bigcup_{i=0}^nB_i=\Omega,
\quad\text{and}\quad
\min_{0\leq i\leq n}\varrho_i=\sum_{i=0}^n1_{B_i}\varrho_i.
\end{equation}
\end{lemma}
\begin{proof}
We proceed by induction on $n$. If $n=1$, we obtain \eqref{e:bx}
by choosing $B_0=[\varrho_0\leq\varrho_1]$ and
$B_1=\complement B_0$. Now assume that the claim is true for $n$,
let $\varrho_{n+1}\in\mathcal{L}(\Omega;\RR)$, and set 
\begin{equation}
\varrho=\min_{0\leq i\leq n}\varrho_i,
\quad
D=[\varrho\leq\varrho_{n+1}],\quad
C_{n+1}=\complement D,
\quad\text{and}\quad
\big(\forall i\in\{0,\ldots,n\}\big)\;\;C_i=B_i\cap D.
\end{equation}
Then $(C_i)_{0\leq i\leq n+1}$ is a family of pairwise disjoint
sets in $\FF$. Additionally,
\begin{equation}
\bigcup_{i=0}^{n+1}C_i
=C_{n+1}\cup\bigcup_{i=0}^nC_i
=\big(\complement D\big)\cup\bigcup_{i=0}^n(B_i\cap D)
=\big(\complement D\big)\cup D
=\Omega
\end{equation}
and
\begin{align}
\min_{0\leq i\leq n+1}\varrho_i
=\min\{\varrho,\varrho_{n+1}\}
=1_D\varrho+1_{\complement D}\varrho_{n+1}
=1_D\sum_{i=0}^n1_{B_i}\varrho_i+1_{C_{n+1}}\varrho_{n+1}
=\sum_{i=0}^{n+1}1_{C_i}\varrho_i,
\end{align}
which concludes the induction argument.
\end{proof}

\begin{lemma}
\label{l:1}
Let $(\Omega,\FF,\mu)$ be a $\sigma$-finite measure space
such that $\mu(\Omega)\neq 0$ and let $\mathcal{R}$ be a nonempty
subset of $\mathcal{L}(\Omega;\RXX)$. Then there exists an
element in $\mathcal{L}(\Omega;\RXX)$, denoted by
$\essinf\mathcal{R}$ and unique up to a set of $\mu$-measure zero,
such that
\begin{equation}
\label{e:od3}
\big(\forall\vartheta\in\mathcal{L}(\Omega;\RXX)\big)
\quad\big[\;(\forall\varrho\in\mathcal{R})\;\;
\vartheta\leq\varrho\,\,\mae\;\big]
\quad\Leftrightarrow\quad
\vartheta\leq\essinf\mathcal{R}\,\,\mae
\end{equation}
Moreover, there exists a sequence $(\varrho_n)_{n\in\NN}$ in
$\mathcal{R}$ such that
$\essinf\mathcal{R}=\inf_{n\in\NN}\varrho_n$.
\end{lemma}
\begin{proof}
Using Assumption~\ref{a:1}\ref{a:1d}, construct
$0<\chi\in\mathcal{L}^1(\Omega;\RR)$ such that
$\int_\Omega\chi d\mu=1$ and define
$\PP\colon\FF\to[0,1]\colon A\mapsto\int_A\chi d\mu$.
Then $(\forall A\in\FF)$ $\mu(A)=0$ $\Leftrightarrow$ $\PP(A)=0$.
Hence, the assertions follow from
\cite[Proposition~II-4-1 and its proof]{Neve70} applied in the
probability space $(\Omega,\FF,\PP)$.
\end{proof}

\begin{lemma}
\label{l:2}
Let $(\Omega,\FF,\mu)$ be a measure space,
let $(\XS,\EuScript{T}_\XS)$ be a Souslin space,
let $z\colon(\Omega,\FF)\to(\XS,\BE_\XS)$ be measurable,
and let $E\in\FF$ be such that $\mu(E)<\pinf$. Then
there exists a sequence $(E_n)_{n\in\NN}$ in $\FF$ such that
\begin{equation}
\big[\;(\forall n\in\NN)\;\;
E_n\subset E\,\,\text{and}\,\,
\overline{z(E_n)}\,\,\text{is compact}\;\big]
\quad\text{and}\quad
\mu(E)=\mu\bigg(\bigcup_{n\in\NN}E_n\bigg).
\end{equation}
\end{lemma}
\begin{proof}
A simple adaptation of the proof of \cite[Lemma~5]{Vala75},
where $(\XS,\EuScript{T}_\XS)$ is a locally convex Souslin
topological vector space. 
\end{proof}

\begin{lemma}
\label{l:3}
Suppose that Assumption~\ref{a:1}\ref{a:1a}--\ref{a:1d} hold.
Let $\psi\colon(\Omega\times\XS,\FF\otimes\BE_\XS)\to\RXX$ be
measurable, let $\mathcal{Z}$ be a nonempty at most countable
subset of $\mathcal{L}(\Omega;\XS)$, and let
$(\Omega_k)_{k\in\NN}$ be an increasing sequence
of finite $\mu$-measure sets in $\FF$ such that
$\bigcup_{k\in\NN}\Omega_k=\Omega$. Define
\begin{equation}
\label{e:8vc}
\mathcal{D}=\bigcup_{z\in\mathcal{Z}}\bigcup_{k\in\NN}
\menge{1_Az}{\FF\ni A\subset\Omega_k\,\,
\text{and}\,\,\overline{z(A)}\,\,\text{is compact}}
\end{equation}
and
\begin{equation}
\label{e:8ki}
\mathcal{R}=\menge{\varrho\in\mathcal{L}^1(\Omega;\RR)}{
(\exi x\in\mathcal{D})\,\,\psi\big(\Cdot,x(\Cdot)\big)\leq
\varrho(\Cdot)\,\,\mae}.
\end{equation}
Suppose that
\begin{equation}
\label{e:h3}
\psi(\Cdot,\mathsf{0})\leq 0.
\end{equation}
Then $\mathcal{R}\neq\emp$ and $\essinf\mathcal{R}
\leq\inf_{z\in\mathcal{Z}}\psi(\Cdot,z(\Cdot))$ $\mae$
\end{lemma}
\begin{proof}
Take $z\in\mathcal{Z}$ and note that
$(\forall A\in\FF)$ $1_Az\in\mathcal{L}(\Omega;\XS)$.
Since $\overline{z(\emp)}=\emp$ is compact,
it results from \eqref{e:8vc} that $0=1_\emp z\in\mathcal{D}$.
Hence, by \eqref{e:h3}, $0\in\mathcal{R}$. Next, thanks
to Assumption~\ref{a:1}\ref{a:1d}, there exists
$\chi\in\mathcal{L}^1(\Omega;\RR)$ such that $\chi>0$.
Let us set
\begin{equation}
\label{e:5h0}
(\forall n\in\NN)\quad A_n=\Omega_n\cap
\big[\psi\big(\Cdot,z(\Cdot)\big)\leq n\chi(\Cdot)\big].
\end{equation}
Lemma~\ref{l:2} asserts that there exists a family
$(A_{n,k})_{(n,k)\in\NN^2}$ in $\FF$ such that
\begin{equation}
\label{e:24r}
(\forall n\in\NN)\quad
\begin{cases}
(\forall k\in\NN)\;\;
A_{n,k}\subset A_n\,\,\text{and}\,\,
\overline{z(A_{n,k})}\,\,\text{is compact}
\\
\displaystyle
\mu(A_n)=\mu\bigg(\bigcup_{k\in\NN}A_{n,k}\bigg).
\end{cases}
\end{equation}
In turn, by \eqref{e:8vc} and \eqref{e:5h0},
\begin{equation}
\label{e:y6}
(\forall n\in\NN)(\forall k\in\NN)\quad
1_{A_{n,k}}z\in\mathcal{D}.
\end{equation}
Define
\begin{equation}
\label{e:c32}
(\forall n\in\NN)(\forall k\in\NN)(\forall m\in\NN)\quad
\varrho_{n,k,m}(\Cdot)=
\max\big\{\psi\big(\Cdot,1_{A_{n,k}}(\Cdot)z(\Cdot)\big),
-m\chi(\Cdot)\big\}.
\end{equation}
Fix temporarily $(n,k,m)\in\NN^3$. We infer from
\eqref{e:24r}, \eqref{e:5h0}, and \eqref{e:h3} that
\begin{align}
(\forall\omega\in\Omega)\quad
\psi\big(\omega,1_{A_{n,k}}(\omega)z(\omega)\big)
&=
\begin{cases}
\psi\big(\omega,z(\omega)\big),
&\text{if}\,\,\omega\in A_{n,k};\\
\psi(\omega,\mathsf{0}),
&\text{otherwise}
\end{cases}
\nonumber\\
&\leq
\begin{cases}
n\chi(\omega),
&\text{if}\,\,\omega\in A_{n,k};\\
0,
&\text{otherwise}
\end{cases}
\nonumber\\
&\leq n\chi(\omega).
\end{align}
Therefore, $-m\chi\leq\varrho_{n,k,m}\leq n\chi$, which entails
that $\varrho_{n,k,m}\in\mathcal{L}^1(\Omega;\RR)$.
In turn, we derive from \eqref{e:c32}, \eqref{e:y6}, and
\eqref{e:8ki} that $\varrho_{n,k,m}\in\mathcal{R}$.
Thus, Lemma~\ref{l:1} guarantees that there exists
$B_{n,k,m}\in\FF$ such that $\mu(B_{n,k,m})=0$ and
\begin{equation}
\label{e:uf}
\big(\forall\omega\in\complement B_{n,k,m}\big)\quad
(\essinf\mathcal{R})(\omega)\leq\varrho_{n,k,m}(\omega).
\end{equation}
Now set
\begin{equation}
\label{e:a34}
A=\bigcap_{(n,k)\in\NN^2}\complement A_{n,k},
\quad
B=\bigcup_{(n,k,m)\in\NN^3}B_{n,k,m},
\quad\text{and}\quad
C=\big[\psi\big(\Cdot,z(\Cdot)\big)<\pinf\big]\cap(A\cup B).
\end{equation}
Then $\mu(B)=0$. Furthermore, since \eqref{e:5h0} yields
$[\psi(\Cdot,z(\Cdot))<\pinf]=\bigcup_{n\in\NN}A_n$, it follows
from \eqref{e:a34} and \eqref{e:24r} that
\begin{equation}
\mu\Big(\big[\psi\big(\Cdot,z(\Cdot)\big)<\pinf\big]\cap A\Big)
\leq\sum_{n\in\NN}\mu(A_n\cap A)
\leq\sum_{n\in\NN}\mu\bigg(A_n\cap
\bigcap_{k\in\NN}\complement A_{n,k}\bigg)
=0.
\end{equation}
Hence, using \eqref{e:a34}, we obtain
\begin{equation}
\label{e:c0}
\mu(C)=0\quad\text{and}\quad
\complement C
=\big[\psi\big(\Cdot,z(\Cdot)\big)=\pinf\big]\cup
\big(\complement A\cap\complement B\big).
\end{equation}
Now suppose that $\omega\in\complement A\cap\complement B$.
Then it follows from \eqref{e:a34} that there exists
$(n,k)\in\NN^2$ such that $\omega\in A_{n,k}\cap\complement B$.
Therefore, we derive from \eqref{e:a34}, \eqref{e:uf}, and
\eqref{e:c32} that
\begin{equation}
\label{e:b9}
(\forall m\in\NN)\quad
(\essinf\mathcal{R})(\omega)
\leq\varrho_{n,k,m}(\omega)
=\max\big\{\psi\big(\omega,1_{A_{n,k}}(\omega)z(\omega)\big),
-m\chi(\omega)\big\}.
\end{equation}
Hence, letting $m\uparrow\pinf$ yields 
$(\essinf\mathcal{R})(\omega)\leq
\psi(\omega,1_{A_{n,k}}(\omega)z(\omega))
=\psi(\omega,z(\omega))$. We have thus shown that
$\essinf\mathcal{R}\leq\psi(\Cdot,z(\Cdot))$ $\mae$
Since $\mathcal{Z}$ is at most countable, the proof is complete.
\end{proof}

\bigskip

\makeatother
\def\prooft{\noindent{\bfseries Proof of Theorem~\ref{t:1}}.
\ignorespaces}
\def\endprooft{\;\:\vbox{\hrule height0.6pt\hbox{%
\vrule height 1.2ex%
width 0.8pt\hskip0.8ex\vrule width 0.8pt}\hrule height 0.6pt}}
\makeatletter

\begin{prooft}
Define
\begin{equation}
\label{e:p9}
\Phi\colon\mathcal{L}(\Omega;\XS)\to\mathcal{L}(\Omega;\RXX)\colon
x\mapsto\varphi\big(\Cdot,x(\Cdot)\big)
\end{equation}
and note that, thanks to Assumption~\ref{a:1}\ref{a:1g},
\begin{equation}
\label{e:o9}
\int_\Omega\inf\varphi(\Cdot,\XS)\,d\mu
\leq\inf_{x\in\XX}\int_\Omega\Phi(x)d\mu
\leq\int_\Omega\Phi(\overline{x})d\mu
<\pinf.
\end{equation}
Hence, the interchange rule \eqref{e:1} holds when
$\inf_{x\in\XX}\int_\Omega\Phi(x)d\mu=\minf$
and we assume henceforth that
\begin{equation}
\label{e:77}
\inf_{x\in\XX}\int_\Omega\Phi(x)d\mu\in\RR.
\end{equation}
Now define
\begin{equation}
\label{e:80}
\vartheta=\max\big\{\Phi(\overline{x}),0\big\}
\end{equation}
and
\begin{equation}
\label{e:72}
\psi\colon\Omega\times\XS\to\RXX\colon
(\omega,\mathsf{x})\mapsto
\begin{cases}
\varphi\big(\omega,\mathsf{x}+\overline{x}(\omega)\big)-
\vartheta(\omega),&\text{if}\,\,\vartheta(\omega)<\pinf;\\
\minf,&\text{if}\,\,\vartheta(\omega)=\pinf.
\end{cases}
\end{equation}
Then we derive from Assumption~\ref{a:1}\ref{a:1g} that
\begin{equation}
\label{e:r1}
\vartheta\in\mathcal{L}^1(\Omega;\RXX)
\end{equation}
and, therefore, that
\begin{equation}
\label{e:r2}
\mu\big([\vartheta=\pinf]\big)=0.
\end{equation}
On the other hand, Assumption~\ref{a:1}\ref{a:1b} ensures that the
mapping $(\Omega\times\XS,\FF\otimes\BE_\XS)\to(\XS,\BE_\XS)\colon
(\omega,\mathsf{x})\mapsto\mathsf{x}+\overline{x}(\omega)$ is
measurable. Thus, it follows from Assumption~\ref{a:1}\ref{a:1f},
\eqref{e:r1}, and \eqref{e:72} that
\begin{equation}
\label{e:f78}
\psi\,\,\text{is $\FF\otimes\BE_\XS$-measurable}.
\end{equation}
At the same time, since
\begin{equation}
\label{e:45}
\inf_{\mathsf{x}\in\XS}\psi(\Cdot,\mathsf{x})
=\inf_{\mathsf{x}\in\XS}
\varphi\big(\Cdot,\mathsf{x}+\overline{x}(\Cdot)\big)-
\vartheta(\Cdot)
=\inf_{\mathsf{x}\in\XS}\varphi(\Cdot,\mathsf{x})-\vartheta(\Cdot)
\end{equation}
and since Assumption~\ref{a:1}\ref{a:1f} yields
$\inf\varphi(\Cdot,\XS)<\pinf$, it results from \ref{t:1i} that
\begin{equation}
\label{e:pd}
\inf\psi(\Cdot,\XS)\in\mathcal{L}(\Omega;\RXX).
\end{equation}
Let us set
\begin{equation}
\Psi\colon\mathcal{L}(\Omega;\XS)\to\mathcal{L}(\Omega;\RXX)\colon
x\mapsto\psi\big(\Cdot,x(\Cdot)\big).
\end{equation}
By \eqref{e:72} and \eqref{e:r2},
\begin{equation}
\label{e:jc}
\big(\forall\omega\in\complement[\vartheta=\pinf]\big)
(\forall x\in\XX)\quad
\big(\Psi(x)\big)(\omega)
=\big(\Phi(x+\overline{x})\big)(\omega)-\vartheta(\omega).
\end{equation}
Hence, upon invoking \eqref{e:r1}, we deduce from
Assumption~\ref{a:1}\ref{a:1e}\&\ref{a:1g} that
\begin{align}
\label{e:15}
\inf_{x\in\XX}\int_\Omega\Psi(x)d\mu
&=\inf_{x\in\XX}\int_\Omega\big(\Phi(x+\overline{x})-
\vartheta\big)d\mu
\nonumber\\
&=\inf_{x\in\XX}\int_\Omega\Phi(x+\overline{x})d\mu-
\int_\Omega\vartheta d\mu
\nonumber\\
&=\inf_{x\in\XX}\int_\Omega\Phi(x)d\mu-\int_\Omega\vartheta d\mu
\end{align}
and, likewise, from \eqref{e:45} that
\begin{equation}
\label{e:14}
\int_\Omega\inf\psi(\Cdot,\XS)\,d\mu
=\int_\Omega\inf\varphi(\Cdot,\XS)\,d\mu-\int_\Omega\vartheta d\mu.
\end{equation}
Now set
\begin{equation}
\label{e:d2}
\mathcal{D}=\bigcup_{n\in\NN}\bigcup_{k\in\NN}
\menge{1_Ax_n}{\FF\ni A\subset\Omega_k
\,\,\text{and}\,\,\overline{x_n(A)}\,\,\text{is compact}}
\end{equation}
and
\begin{equation}
\label{e:pui}
\mathcal{R}=\menge{\varrho\in\mathcal{L}^1(\Omega;\RR)}{
(\exi x\in\mathcal{D})\,\,\Psi(x)\leq\varrho\,\,\mae},
\end{equation}
and note that \ref{t:1iib} states that
\begin{equation}
\label{e:d0}
\mathcal{D}\subset\XX.
\end{equation}
Using \eqref{e:72} and \eqref{e:80}, we infer from
Lemma~\ref{l:3} applied to $\mathcal{Z}=\{x_n\}_{n\in\NN}$ that
$\essinf\mathcal{R}\leq\inf_{n\in\NN}\Psi(x_n)$ $\mae$ In turn, we
derive from \eqref{e:jc}, \ref{t:1iia}, and \eqref{e:45}
that
\begin{equation}
\essinf\mathcal{R}
\leq\inf_{n\in\NN}\Psi(x_n)
=\inf_{n\in\NN}\Phi(x_n+\overline{x})-\vartheta
=\inf\varphi(\Cdot,\XS)-\vartheta
=\inf\psi(\Cdot,\XS)\,\,\mae
\end{equation}
On the other hand, \eqref{e:pui} implies that
$(\forall\varrho\in\mathcal{R})$
$\inf\psi(\Cdot,\XS)\leq\varrho(\Cdot)$ $\mae$
Hence, \eqref{e:pd} and Lemma~\ref{l:1} guarantee that
$\inf\psi(\Cdot,\XS)\leq\essinf\mathcal{R}$ $\mae$
Altogether, $\essinf\mathcal{R}=\inf\psi(\Cdot,\XS)$ $\mae$
Thus, we deduce from Lemma~\ref{l:1} that there exists a sequence
$(\varrho_n)_{n\in\NN}$ in $\mathcal{R}$ such that
\begin{equation}
\label{e:g8}
\inf_{n\in\NN}\varrho_n(\Cdot)=\inf\psi(\Cdot,\XS)\,\,\mae
\end{equation}
For every $n\in\NN$, it follows from \eqref{e:pui} and
\eqref{e:d2} that there exist $\ell_n\in\NN$,
$k_n\in\NN$, and $\FF\ni A_n\subset\Omega_{k_n}$ such that
\begin{equation}
\label{e:o4}
\overline{x_{\ell_n}(A_n)}\,\,\text{is compact}
\quad\text{and}\quad
\Psi\big(1_{A_n}x_{\ell_n}\big)\leq\varrho_n\,\,\mae
\end{equation}
Let us set
\begin{equation}
\label{e:c5}
(\forall n\in\NN)\quad\chi_n=\min_{0\leq i\leq n}\varrho_i.
\end{equation}
Fix temporarily $n\in\NN$. Lemma~\ref{l:8} asserts that there
exists a family $(B_{n,i})_{0\leq i\leq n}$ in $\FF$ such that
\begin{equation}
\label{e:bi6}
(B_{n,i})_{0\leq i\leq n}\,\,\text{are pairwise disjoint},
\quad\bigcup_{i=0}^nB_{n,i}=\Omega,\quad\text{and}\quad
\chi_n=\sum_{i=0}^n1_{B_{n,i}}\varrho_i.
\end{equation}
Now set
\begin{equation}
\label{e:y8}
y_n=\sum_{i=0}^n1_{A_i\cap B_{n,i}}x_{\ell_i}.
\end{equation}
For every $i\in\{0,\ldots,n\}$, since
$A_i\cap B_{n,i}\subset A_i\subset\Omega_{k_i}$, \eqref{e:o4}
implies that $\overline{x_{\ell_i}(A_i\cap B_{n,i})}$ is compact
and, therefore, \eqref{e:d2} and \eqref{e:d0} yield
$1_{A_i\cap B_{n,i}}x_{\ell_i}\in\mathcal{D}\subset\XX$.
Consequently, \eqref{e:y8} and Assumption~\ref{a:1}\ref{a:1e}
ensure that $y_n\in\XX$. At the same time, we derive from
\eqref{e:y8}, \eqref{e:bi6}, and \eqref{e:o4} that
\begin{equation}
\label{e:7gf}
\Psi(y_n)
=\sum_{i=0}^n1_{B_{n,i}}\Psi\big(1_{A_i}x_{\ell_i}\big)
\leq\sum_{i=0}^n1_{B_{n,i}}\varrho_i
=\chi_n\,\,\mae
\end{equation}
Therefore, since $y_n\in\XX$,
\begin{equation}
\label{e:g7}
\inf_{x\in\XX}\int_\Omega\Psi(x)d\mu
\leq\int_\Omega\Psi(y_n)d\mu
\leq\int_\Omega\chi_nd\mu.
\end{equation}
On the other hand, it results from
\eqref{e:15}, \eqref{e:77}, and \eqref{e:r1} that
$\inf_{x\in\XX}\int_\Omega\Psi(x)d\mu\in\RR$.
Thus, since $\chi_n\downarrow\inf_{i\in\NN}\varrho_i(\Cdot)=
\inf\psi(\Cdot,\XS)$ $\mae$
by virtue of \eqref{e:c5} and \eqref{e:g8}, \eqref{e:g7} and
the monotone convergence theorem
\cite[Theorem~2.8.2 and Corollary~2.8.6]{Boga07}
entail that
\begin{equation}
\inf_{x\in\XX}\int_\Omega\Psi(x)d\mu
\leq\lim\int_\Omega\chi_nd\mu
=\int_\Omega\lim\chi_n\,d\mu
=\int_\Omega\inf\psi(\Cdot,\XS)\,d\mu.
\end{equation}
Consequently, since $\int_\Omega\inf\psi(\Cdot,\XS)\,d\mu\leq
\inf_{x\in\XX}\int_\Omega\Psi(x)d\mu$, we conclude that
\begin{equation}
\inf_{x\in\XX}\int_\Omega\Psi(x)d\mu
=\int_\Omega\inf\psi(\Cdot,\XS)\,d\mu.
\end{equation}
In view of \eqref{e:15}, \eqref{e:14}, and \eqref{e:r1},
the proof is complete.
\end{prooft}

\begin{remark}
\label{r:1}
Replacing $\varphi$ by $-\varphi$ in items \ref{a:1f} and
\ref{a:1g} of Assumption~\ref{a:1} and in Theorem~\ref{t:1}
provides conditions under which
\begin{equation}
\sup_{x\in\XX}\int_\Omega\varphi\big(\omega,x(\omega)\big)
\mu(d\omega)=
\int_\Omega\sup_{\mathsf{x}\in\XS}\varphi(\omega,\mathsf{x})\,
\mu(d\omega),
\end{equation}
with the convention that, given a measurable function
$\varrho\colon(\Omega,\FF)\to\RXX$, $\int_\Omega\varrho d\mu$ is
the usual Lebesgue integral, except when the Lebesgue integral
$\int_\Omega\min\{\varrho,0\}d\mu$ is $\minf$, in which case
$\int_\Omega\varrho d\mu=\minf$.
\end{remark}

\begin{remark}
\label{r:8}
In Theorem~\ref{t:1}, suppose that
$\inf_{x\in\XX}\int_\Omega\varphi(\Cdot,x(\Cdot))
d\mu>\minf$ and let $z\in\XX$. Then
\begin{equation}
\label{e:1x}
\int_\Omega\varphi\big(\omega,z(\omega)\big)
\mu(d\omega)
=\min_{x\in\XX}\int_\Omega\varphi\big(\omega,x(\omega)\big)
\mu(d\omega)
\quad\Leftrightarrow\quad
\varphi\big(\Cdot,z(\Cdot)\big)=\min\varphi(\Cdot,\XS)\,\,\mae
\end{equation}
\end{remark}

\section{Compliant spaces and normal integrands}
\label{sec:4}

The objective of this section is to develop tools to convert the
interchange principle of Theorem~\ref{t:1} into interchange rules
formulated in terms of explicit conditions on the ambient space
$\XX$ and the integrand $\varphi$. Our framework hinges on a
notion of compliant spaces and a notion of normal integrands in an
extended sense.

\subsection{Compliant spaces}
\label{sec:41}

We introduce the following notion of a compliant space, which
generalizes and unifies the notions of decomposability employed in
the interchange rules of
\cite{Penn23,Perk18,Rock71,Roc76k,Rock09,Shap21,Vala75}.

\begin{definition}[compliance]
\label{d:1}
Suppose that Assumption~\ref{a:1}\ref{a:1a}--\ref{a:1e} holds. Then
$\XX$ is \emph{compliant} if, for every $A\in\FF$ such that
$\mu(A)<\pinf$ and every $z\in\mathcal{L}(\Omega;\XS)$ such
that $\overline{z(A)}$ is compact, $1_Az\in\XX$.
\end{definition}

\begin{proposition}
\label{p:10}
Suppose that Assumption~\ref{a:1}\ref{a:1a}--\ref{a:1e} holds,
together with one of the following:
\begin{enumerate}
\item
\label{p:10i}
$(\XS,\EuScript{T}_\XS)$ is a Souslin topological vector space and,
for every $A\in\FF$ such that $\mu(A)<\pinf$ and every
$z\in\mathcal{L}(\Omega;\XS)$ such that $z(A)$ is
$\EuScript{T}_\XS$-bounded (in the sense that, for every
neighborhood $\mathsf{V}\in\EuScript{T}_\XS$
of $\mathsf{0}$, there exists $\alpha\in\RPP$ such that
$z(A)\subset\bigcap_{\beta>\alpha}\beta\mathsf{V}$
\cite{Rudi91}), $1_Az\in\XX$.
\item
\label{p:10ii}
$\XS$ is a separable Banach space with strong topology
$\EuScript{T}_\XS$ and, for every $A\in\FF$ such that
$\mu(A)<\pinf$ and every
$z\in\mathcal{L}^\infty(\Omega;\XS)$, $1_Az\in\XX$.
\item
\label{p:10iii}
$\XS$ is a separable Banach space with strong topology
$\EuScript{T}_\XS$, $\mu(\Omega)<\pinf$, and
$\mathcal{L}^\infty(\Omega;\XS)\subset\XX$.
\item
\label{p:10iv}
$\XS$ is a separable Banach space with strong topology
$\EuScript{T}_\XS$ and $\XX$ is \emph{Rockafellar-decomposable}
\cite{Rock71} in the sense that, for every $A\in\FF$ such that
$\mu(A)<\pinf$, every $z\in\mathcal{L}^\infty(\Omega;\XS)$,
and every $x\in\XX$, $1_Az+1_{\complement A}x\in\XX$.
\item
\label{p:10v}
$(\XS,\EuScript{T}_\XS)$ is a Souslin locally convex topological
vector space and $\XX$ is \emph{Valadier-decomposable}
\cite{Vala75} in the sense that, for every $A\in\FF$ such that
$\mu(A)<\pinf$, every $z\in\mathcal{L}(\Omega;\XS)$ such
that $\overline{z(A)}$ is compact, and every $x\in\XX$,
$1_Az+1_{\complement A}x\in\XX$.
\item
\label{p:10vi}
$\XS$ is the standard Euclidean space $\RR^N$
and, for every $A\in\FF$ such that $\mu(A)<\pinf$ and every
$z\in\mathcal{L}^\infty(\Omega;\XS)$, $1_Az\in\XX$.
\end{enumerate}
Then $\XX$ is compliant.
\end{proposition}
\begin{proof}
\ref{p:10i}:
Let $A\in\FF$ be such that $\mu(A)<\pinf$ and let
$z\in\mathcal{L}(\Omega;\XS)$ be such that $\overline{z(A)}$
is compact. It results from \cite[Theorem~1.15(b)]{Rudi91} that
$z(A)$ is $\EuScript{T}_\XS$-bounded. Thus $1_Az\in\XX$.

\ref{p:10iii}$\Rightarrow$\ref{p:10ii}$\Rightarrow$\ref{p:10i}:
Clear.

\ref{p:10iv}$\Rightarrow$\ref{p:10ii}: Clear.

\ref{p:10v}: Clear.

\ref{p:10vi}$\Rightarrow$\ref{p:10ii}: Clear.
\end{proof}

\subsection{Normal integrands}
\label{sec:42}

We introduce a notion of a normal integrand which unifies and
extends those of \cite{Roc68a,Rock71,Roc76k,Vala75}.

\begin{definition}[normality]
\label{d:n}
Let $(\XS,\EuScript{T}_\XS)$ be a Souslin space, let
$(\Omega,\FF)$ be a measurable space,
let $\varphi\colon(\Omega\times\XS,\FF\otimes\BE_\XS)\to\RXX$ be
measurable, and equip $\XS\times\RR$ with the topology
$\EuScript{T}_\XS\boxtimes\EuScript{T}_\RR$.
Then $\varphi$ is a \emph{normal integrand} if there
exist sequences $(x_n)_{n\in\NN}$ in
$\mathcal{L}(\Omega;\XS)$ and $(\varrho_n)_{n\in\NN}$ in
$\mathcal{L}(\Omega;\RR)$ such that
\begin{equation}
(\forall\omega\in\Omega)\quad
\big\{\big(x_n(\omega),\varrho_n(\omega)\big)\big\}_{n\in\NN}
\subset\epi\varphi_\omega
\quad\text{and}\quad
\overline{\epi\varphi_\omega}=
\overline{\big\{\big(x_n(\omega),
\varrho_n(\omega)\big)\big\}_{n\in\NN}}.
\end{equation}
\end{definition}

The following theorem furnishes examples of normal integrands.

\begin{theorem}
\label{t:3}
Let $(\XS,\EuScript{T}_\XS)$ be a Souslin space,
let $(\Omega,\FF)$ be a measurable space, and
let $\varphi\colon\Omega\times\XS\to\RXX$ be such that
$(\forall\omega\in\Omega)$ $\epi\varphi_\omega\neq\emp$.
Suppose that one of the following holds:
\begin{enumerate}
\item
\label{t:3i}
$\varphi$ is $\FF\otimes\BE_\XS$-measurable and
one of the following is satisfied:
\begin{enumerate}
\item
\label{t:3ia}
There exists a measure $\mu$ such that
$(\Omega,\FF,\mu)$ is complete and $\sigma$-finite.
\item
\label{t:3ib}
$\Omega$ is a Borel subset of $\RR^M$ and $\FF$ is the
associated Lebesgue $\sigma$-algebra.
\item
\label{t:3ic}
For every $\omega\in\Omega$, there exists
$\boldsymbol{\mathsf{V}}_\omega\in
\EuScript{T}_\XS\boxtimes\EuScript{T}_\RR$ such that
$\boldsymbol{\mathsf{V}}_\omega\subset\epi\varphi_\omega$
and $\overline{\boldsymbol{\mathsf{V}}_\omega}
=\overline{\epi\varphi_\omega}$.
\item
\label{t:3id}
The functions $(\varphi_\omega)_{\omega\in\Omega}$ are upper
semicontinuous.
\end{enumerate}
\item
\label{t:3ii}
The functions $(\varphi(\Cdot,\mathsf{x}))_{\mathsf{x}\in\XS}$ are
$\FF$-measurable and one of the following is satisfied:
\begin{enumerate}
\item
\label{t:3iia}
$(\XS,\EuScript{T}_\XS)$ is metrizable and, for every
$\omega\in\Omega$, there exists $\boldsymbol{\mathsf{V}}_\omega\in
\EuScript{T}_\XS\boxtimes\EuScript{T}_\RR$ such that
$\boldsymbol{\mathsf{V}}_\omega\subset\epi\varphi_\omega
=\overline{\boldsymbol{\mathsf{V}}_\omega}$.
\item
\label{t:3iib}
$(\XS,\EuScript{T}_\XS)$ is a Fr\'echet space and, for every
$\omega\in\Omega$, $\varphi_\omega\in\Gamma_0(\XS)$ and
$\intdom\varphi_\omega\neq\emp$.
\item
\label{t:3iic}
$(\XS,\EuScript{T}_\XS)$ is the standard Euclidean line $\RR$
and, for every $\omega\in\Omega$, $\varphi_\omega\in\Gamma_0(\RR)$
and $\dom\varphi_\omega$ is not a singleton.
\end{enumerate}
\item
\label{t:3iii}
$(\XS,\EuScript{T}_\XS)$ is a regular Souslin space,
the functions $(\varphi_\omega)_{\omega\in\Omega}$ are continuous,
and the functions $(\varphi(\Cdot,\mathsf{x}))_{\mathsf{x}\in\XS}$
are $\FF$-measurable.
\item
\label{t:3iv}
For some separable Fr\'echet space $(\YS,\EuScript{T}_\YS)$,
$\XS=(\YS,\EuScript{T}_\YS)^*$, $\EuScript{T}_\XS$ is the weak
topology, the functions
$(\varphi_\omega)_{\omega\in\Omega}$ are $\EuScript{T}_\XS$-lower
semicontinuous, and one of the following is satisfied:
\begin{enumerate}
\item
\label{t:3iva}
For every closed subset $\boldsymbol{\mathsf{C}}$ of
$(\XS\times\RR,\EuScript{T}_\XS\boxtimes\EuScript{T}_\RR)$,
$\menge{\omega\in\Omega}{\boldsymbol{\mathsf{C}}
\cap\epi\varphi_\omega\neq\emp}\in\FF$.
\item
\label{t:3ivb}
$(\Omega,\EuScript{T}_\Omega)$ is a Hausdorff topological space,
$\FF=\BE_\Omega$, and
$\varphi$ is $\EuScript{T}_\Omega\boxtimes\EuScript{T}_\XS$-lower
semicontinuous.
\item
\label{t:3ivc}
$(\Omega,\EuScript{T}_\Omega)$ is a Lusin space,
$\FF=\BE_\Omega$, and $\varphi$ is $\FF\otimes\BE_\XS$-measurable.
\end{enumerate}
\item
\label{t:3v}
$\XS$ is a separable reflexive Banach space, $\EuScript{T}_\XS$ is
the weak topology,
$(\Omega,\EuScript{T}_\Omega)$ is a Hausdorff topological space,
$\FF=\BE_\Omega$, the functions
$(\varphi_\omega)_{\omega\in\Omega}$ are
$\EuScript{T}_\XS$-lower semicontinuous, and one of the following
is satisfied:
\begin{enumerate}
\item
\label{t:3va}
$\varphi$ is $\EuScript{T}_\Omega\boxtimes\EuScript{T}_\XS$-lower
semicontinuous.
\item
\label{t:3vb}
$(\Omega,\EuScript{T}_\Omega)$ is a Lusin space and
$\varphi$ is $\FF\otimes\BE_\XS$-measurable.
\end{enumerate}
\item
\label{t:3vi}
$(\XS,\EuScript{T}_\XS)$ is the standard Euclidean space $\RR^N$,
$\Omega$ is a Borel subset of $\RR^M$,
$\FF=\BE_\Omega$, $\varphi$ is $\FF\otimes\BE_\XS$-measurable, and
the functions $(\varphi_\omega)_{\omega\in\Omega}$ are lower
semicontinuous.
\item
\label{t:3vii}
$(\XS,\EuScript{T}_\XS)$ is a Polish space, the functions
$(\varphi_\omega)_{\omega\in\Omega}$ are lower semicontinuous, and
one of the following is satisfied:
\begin{enumerate}
\item
\label{t:3viia}
For every $\boldsymbol{\mathsf{V}}\in
\EuScript{T}_\XS\boxtimes\EuScript{T}_\RR$,
$\menge{\omega\in\Omega}{\boldsymbol{\mathsf{V}}\cap
\epi\varphi_\omega\neq\emp}\in\FF$.
\item
\label{t:3viib}
$(\XS,\EuScript{T}_\XS)$ is the standard Euclidean space $\RR^N$
and, for every closed subset $\boldsymbol{\mathsf{C}}$ of
$\XS\times\RR$, $\menge{\omega\in\Omega}{
\boldsymbol{\mathsf{C}}\cap\epi\varphi_\omega\neq\emp}\in\FF$.
\end{enumerate}
\item
\label{t:3viii}
There exists a measurable function
$\mathsf{f}\colon(\XS,\BE_\XS)\to\RXX$ such that
$(\forall\omega\in\Omega)$ $\varphi_\omega=\mathsf{f}$.
\end{enumerate}
Then $\varphi$ is normal.
\end{theorem}
\begin{proof}
Set $\boldsymbol{G}=\menge{(\omega,\mathsf{x},\xi)\in
\Omega\times\XS\times\RR}{\varphi(\omega,\mathsf{x})\leq\xi}$.
Then
\begin{equation}
\label{e:s03}
\boldsymbol{G}=\menge{(\omega,\mathsf{x},\xi)\in
\Omega\times\XS\times\RR}{(\mathsf{x},\xi)\in\epi\varphi_\omega}.
\end{equation}
Further, \cite[Lemma~6.4.2(i)]{Boga07} yields
\begin{equation}
\label{e:s04}
\varphi\,\,\text{is $\FF\otimes\BE_\XS$-measurable}
\quad\Leftrightarrow\quad
\boldsymbol{G}\in\FF\otimes\BE_\XS\otimes\BE_\RR
=\FF\otimes\BE_{\XS\times\RR}.
\end{equation}
We also note that $(\XS\times\RR,
\EuScript{T}_\XS\boxtimes\EuScript{T}_\RR)$ is a Souslin space
\cite[Proposition~IX.6.7]{Bour74}.

\ref{t:3ia}:
Applying \cite[Theorem~III.22]{Cast77} to the mapping
$\Upsilon\colon\Omega\to
2^{\XS\times\RR}\colon\omega\mapsto\epi\varphi_\omega$, we deduce
from \eqref{e:s03} and \eqref{e:s04} that there exist a sequence
$(x_n)_{n\in\NN}$ of mappings from $\Omega$ to $\XS$ and a
sequence $(\varrho_n)_{n\in\NN}$ of functions from $\Omega$ to
$\RR$ such that
\begin{equation}
\label{e:10ds}
(\forall n\in\NN)\quad
(\Omega,\FF)\to(\XS\times\RR,\BE_{\XS\times\RR})\colon
\omega\mapsto\big(x_n(\omega),\varrho_n(\omega)\big)
\,\,\text{is measurable}
\end{equation}
and
\begin{equation}
(\forall\omega\in\Omega)\quad
\big\{\big(x_n(\omega),\varrho_n(\omega)\big)\big\}_{n\in\NN}
\subset\Upsilon(\omega)
\quad\text{and}\quad
\overline{\Upsilon(\omega)}=
\overline{\big\{\big(x_n(\omega),
\varrho_n(\omega)\big)\big\}_{n\in\NN}}.
\end{equation}
Moreover, since $\BE_{\XS\times\RR}=\BE_\XS\otimes\BE_\RR$
\cite[Lemma~6.4.2(i)]{Boga07},
it follows from \eqref{e:10ds} that, for every $n\in\NN$,
$x_n\colon(\Omega,\FF)\to(\XS,\BE_\XS)$ and
$\varrho_n\colon(\Omega,\FF)\to(\RR,\BE_\RR)$
are measurable. Altogether, $\varphi$ is normal.

\ref{t:3ib}$\Rightarrow$\ref{t:3ia}:
Take $\mu$ to be the Lebesgue measure on $\Omega$.

\ref{t:3ic}:
Let $\{(\mathsf{x}_n,\xi_n)\}_{n\in\NN}$ be a dense set in
$(\XS\times\RR,\EuScript{T}_\XS\boxtimes\EuScript{T}_\RR)$ and
define
\begin{equation}
\label{e:rc0}
(\forall n\in\NN)\quad
\Omega_n=\big[\varphi(\Cdot,\mathsf{x}_n)\leq\xi_n\big].
\end{equation}
On the one hand, the $\FF\otimes\BE_\XS$-measurability of
$\varphi$ ensures
that $(\forall n\in\NN)$ $\Omega_n\in\FF$.
On the other hand, for every $\omega\in\Omega$,
since $\boldsymbol{\mathsf{V}}_\omega$ is open,
there exists $n\in\NN$ such that
$(\mathsf{x}_n,\xi_n)\in\boldsymbol{\mathsf{V}}_\omega
\subset\epi\varphi_\omega$, which yields
$\omega\in\Omega_n$ and thus
$\Omega=\bigcup_{k\in\NN}\Omega_k$.
This yields a sequence $(\Theta_n)_{n\in\NN}$ of pairwise disjoint
sets in $\FF$ such that
\begin{equation}
\label{e:tf0}
\Theta_0=\Omega_0,\quad
\bigcup_{n\in\NN}\Theta_n=\Omega,\quad\text{and}\quad
(\forall n\in\NN)\;\;\Theta_n\subset\Omega_n.
\end{equation}
For every $\omega\in\Omega$, there exists a unique 
$n_\omega\in\NN$ such that $\omega\in\Theta_{n_\omega}$.
Now define
\begin{equation}
z\colon\Omega\to\XS\colon\omega\mapsto\mathsf{x}_{n_\omega}
\quad\text{and}\quad
\vartheta\colon\Omega\to\RR\colon\omega\mapsto\xi_{n_\omega}.
\end{equation}
Then
\begin{equation}
(\forall\mathsf{V}\in\EuScript{T}_\XS)\quad
z^{-1}(\mathsf{V})=\bigcup_{\substack{n\in\NN\\
\mathsf{x}_n\in\mathsf{V}}}\Theta_n\in\FF,
\end{equation}
which implies that $z\in\mathcal{L}(\Omega;\XS)$. Likewise,
$\vartheta\in\mathcal{L}(\Omega;\RR)$. Next, define
\begin{equation}
\label{e:xxi0}
(\forall n\in\NN)\quad
x_n\colon\Omega\to\XS\colon\omega\mapsto
\begin{cases}
\mathsf{x}_n,&\text{if}\,\,\omega\in\Omega_n;\\
z(\omega),&\text{if}\,\,\omega\in\complement\Omega_n
\end{cases}
\end{equation}
and
\begin{equation}
\label{e:xxi1}
(\forall n\in\NN)\quad
\varrho_n\colon\Omega\to\RR\colon\omega\mapsto
\begin{cases}
\xi_n,&\text{if}\,\,\omega\in\Omega_n;\\
\vartheta(\omega),
&\text{if}\,\,\omega\in\complement\Omega_n.
\end{cases}
\end{equation}
Then $(x_n)_{n\in\NN}$ and $(\varrho_n)_{n\in\NN}$ are sequences
in $\mathcal{L}(\Omega;\XS)$ and
$\mathcal{L}(\Omega;\RR)$, respectively.
Moreover, we deduce from \eqref{e:xxi0}, \eqref{e:xxi1},
\eqref{e:rc0}, and \eqref{e:tf0} that
\begin{equation}
(\forall\omega\in\Omega)(\forall n\in\NN)\quad
\big(x_n(\omega),\varrho_n(\omega)\big)\in\epi\varphi_\omega.
\end{equation}
On the other hand, for every $\omega\in\Omega$,
since $\{(\mathsf{x}_n,\xi_n)\}_{n\in\NN}$ is
dense in $(\XS\times\RR,\EuScript{T}_\XS\boxtimes\EuScript{T}_\RR)$
and since
$\boldsymbol{\mathsf{V}}_\omega$ is open, we infer from
\eqref{e:xxi0}, \eqref{e:xxi1}, and \eqref{e:rc0} that
\begin{equation}
\overline{\big\{\big(x_n(\omega),
\varrho_n(\omega)\big)\big\}_{n\in\NN}}
=\overline{\big\{(\mathsf{x}_n,\xi_n)\big\}_{n\in\NN}\cap
\epi\varphi_\omega}
\supset\overline{\big\{(\mathsf{x}_n,\xi_n)\big\}_{n\in\NN}\cap
\boldsymbol{\mathsf{V}}_\omega}
=\overline{\boldsymbol{\mathsf{V}}_\omega}
=\overline{\epi\varphi_\omega}.
\end{equation}
Consequently, $\varphi$ is normal.

\ref{t:3id}$\Rightarrow$\ref{t:3ic}:
Set $(\forall\omega\in\Omega)$
$\boldsymbol{\mathsf{V}}_\omega=\menge{(\mathsf{x},\xi)\in
\XS\times\RR}{\varphi(\omega,\mathsf{x})<\xi}$.
Now fix $\omega\in\Omega$ and
$(\mathsf{x},\xi)\in\epi\varphi_\omega$.
Since the sequence $(\mathsf{x},\xi+2^{-n})_{n\in\NN}$
lies in $\boldsymbol{\mathsf{V}}_\omega$ and
$(\mathsf{x},\xi+2^{-n})\to(\mathsf{x},\xi)$, we obtain
$(\mathsf{x},\xi)\in\overline{\boldsymbol{\mathsf{V}}_\omega}$.
Hence $\overline{\boldsymbol{\mathsf{V}}_\omega}=
\overline{\epi\varphi_\omega}$.
At the same time, the upper semicontinuity of $\varphi_\omega$
guarantees that $\boldsymbol{\mathsf{V}}_\omega$ is open.

\ref{t:3iia}$\Rightarrow$\ref{t:3ic}:
It suffices to show that $\varphi$ is
$\FF\otimes\BE_\XS$-measurable. Let
$\{(\mathsf{x}_n,\xi_n)\}_{n\in\NN}$ be dense in
$(\XS\times\RR,\EuScript{T}_\XS\boxtimes\EuScript{T}_\RR)$,
let $\boldsymbol{\mathsf{V}}\in
\EuScript{T}_\XS\boxtimes\EuScript{T}_\RR$,
and set $\mathbb{K}=\menge{n\in\NN}{
(\mathsf{x}_n,\xi_n)\in\boldsymbol{\mathsf{V}}}$. Then
\begin{equation}
\label{e:vd9}
\overline{\{(\mathsf{x}_n,\xi_n)\}_{n\in\mathbb{K}}}
=\overline{\{(\mathsf{x}_n,\xi_n)\}_{n\in\NN}\cap
\boldsymbol{\mathsf{V}}}
=\overline{\boldsymbol{\mathsf{V}}}.
\end{equation}
Suppose that there exists $\omega\in\Omega$ such that
\begin{equation}
\label{e:tx}
\boldsymbol{\mathsf{V}}\cap\epi\varphi_\omega\neq\emp
\quad\text{and}\quad
(\forall n\in\mathbb{K})\;\;
(\mathsf{x}_n,\xi_n)\notin\epi\varphi_\omega.
\end{equation}
Since $\boldsymbol{\mathsf{V}}$ is open and
$\overline{\boldsymbol{\mathsf{V}}_\omega}=\epi\varphi_\omega$,
there exists $(\mathsf{y},\eta)
\in\boldsymbol{\mathsf{V}}\cap\boldsymbol{\mathsf{V}}_\omega$.
Therefore, we infer from \eqref{e:vd9} that there exists a subnet
$(\mathsf{x}_{k(b)},\xi_{k(b)})_{b\in B}$ of
$(\mathsf{x}_n,\xi_n)_{n\in\mathbb{K}}$ such that
$(\mathsf{x}_{k(b)},\xi_{k(b)})\to(\mathsf{y},\eta)$.
This and \eqref{e:tx} force
$(\mathsf{y},\eta)
\in\overline{\complement\epi\varphi_\omega}
=\overline{\complement
\overline{\boldsymbol{\mathsf{V}}_\omega}}
=\complement\inte
\overline{\boldsymbol{\mathsf{V}}_\omega}$,
which is in contradiction with the inclusion
$(\mathsf{y},\eta)\in\boldsymbol{\mathsf{V}}_\omega$.
Hence, the $\FF$-measurability of the functions
$(\varphi(\Cdot,\mathsf{x}))_{\mathsf{x}\in\XS}$ yields
\begin{equation}
\menge{\omega\in\Omega}{
\boldsymbol{\mathsf{V}}\cap\epi\varphi_\omega\neq\emp}
=\bigcup_{n\in\mathbb{K}}\menge{\omega\in\Omega}{
(\mathsf{x}_n,\xi_n)\in\epi\varphi_\omega}
=\bigcup_{n\in\mathbb{K}}
\big[\varphi(\Cdot,\mathsf{x}_n)\leq\xi_n\big]
\in\FF.
\end{equation}
Therefore, since
$(\XS\times\RR,\EuScript{T}_\XS\boxtimes\EuScript{T}_\RR)$ is a
separable metrizable space
and the sets $(\epi\varphi_\omega)_{\omega\in\Omega}$ are closed,
\cite[Theorem~3.5(i)]{Himm75} and \eqref{e:s03} imply that
$\boldsymbol{G}\in\FF\otimes\BE_{\XS\times\RR}$.
Consequently, \eqref{e:s04} asserts that $\varphi$ is
$\FF\otimes\BE_\XS$-measurable.

\ref{t:3iib}$\Rightarrow$\ref{t:3iia}:
Set $(\forall\omega\in\Omega)$
$\boldsymbol{\mathsf{V}}_\omega=\inte\epi\varphi_\omega$.
For every $\omega\in\Omega$, the assumption ensures that
$\epi\varphi_\omega$ is closed and convex, and that
$\boldsymbol{\mathsf{V}}_\omega\neq\emp$
\cite[Theorem~2.2.20 and Corollary~2.2.10]{Zali02}.
Thus \cite[Theorem~1.1.2(iv)]{Zali02} yields
$(\forall\omega\in\Omega)$
$\epi\varphi_\omega=\overline{\boldsymbol{\mathsf{V}}_\omega}$.

\ref{t:3iic}$\Rightarrow$\ref{t:3iib}:
Clear.

\ref{t:3iii}:
It results from \cite{Sain76} that there exists a topology
$\widetilde{\EuScript{T}_\XS}$ on $\XS$ such that
\begin{equation}
\label{e:8dy}
\EuScript{T}_\XS\subset\widetilde{\EuScript{T}_\XS}
\end{equation}
and
\begin{equation}
\label{e:8dx}
\big(\XS,\widetilde{\EuScript{T}_\XS}\big)\,\,
\text{is a metrizable Souslin space}.
\end{equation}
Set $(\forall\omega\in\Omega)$
$\boldsymbol{\mathsf{V}}_\omega=\menge{(\mathsf{x},\xi)\in
\XS\times\RR}{\varphi(\omega,\mathsf{x})<\xi}$.
Then, since \eqref{e:8dy} implies that
\begin{equation}
\label{e:gf}
(\forall\omega\in\Omega)\quad
\varphi_\omega\,\,
\text{is $\widetilde{\EuScript{T}_\XS}$-continuous},
\end{equation}
it follows that
\begin{equation}
\label{e:r0d}
(\forall\omega\in\Omega)\quad
\boldsymbol{\mathsf{V}}_\omega\in
\widetilde{\EuScript{T}_\XS}\boxtimes\EuScript{T}_\RR
\quad\text{and}\quad
\overline{\boldsymbol{\mathsf{V}}_\omega}^{
\widetilde{\EuScript{T}_\XS}\boxtimes\EuScript{T}_\RR}
=\overline{\epi\varphi_\omega}^{\widetilde{\EuScript{T}_\XS}
\boxtimes\EuScript{T}_\RR}
=\epi\varphi_\omega.
\end{equation}
On the other hand, we derive from \eqref{e:8dx}, \eqref{e:8dy},
and \cite[Corollary~2, p.~101]{Schw73} that
the Borel $\sigma$-algebra of $(\XS,\widetilde{\EuScript{T}_\XS})$
is $\BE_\XS$.
Altogether, applying \ref{t:3iia} to the metrizable Souslin space
$(\XS,\widetilde{\EuScript{T}_\XS})$, 
we deduce that $\varphi$ is $\FF\otimes\BE_\XS$-measurable and
that there exist sequences $(x_n)_{n\in\NN}$ in
$\mathcal{L}(\Omega;\XS)$ and $(\varrho_n)_{n\in\NN}$ in
$\mathcal{L}(\Omega;\RR)$ such that
\begin{equation}
(\forall\omega\in\Omega)\quad
\big\{\big(x_n(\omega),\varrho_n(\omega)\big)\big\}_{n\in\NN}
\subset\epi\varphi_\omega
\quad\text{and}\quad
\overline{\epi\varphi_\omega}^{
\widetilde{\EuScript{T}_\XS}\boxtimes\EuScript{T}_\RR}=
\overline{\big\{\big(x_n(\omega),
\varrho_n(\omega)\big)\big\}_{n\in\NN}}^{
\widetilde{\EuScript{T}_\XS}\boxtimes\EuScript{T}_\RR}.
\end{equation}
Hence, by \eqref{e:8dy} and \eqref{e:r0d},
\begin{align}
\overline{\big\{\big(x_n(\omega),
\varrho_n(\omega)\big)\big\}_{n\in\NN}}
\supset\overline{\big\{\big(x_n(\omega),
\varrho_n(\omega)\big)\big\}_{n\in\NN}}^{
\widetilde{\EuScript{T}_\XS}\boxtimes\EuScript{T}_\RR}
=\overline{\epi\varphi_\omega}^{
\widetilde{\EuScript{T}_\XS}\boxtimes\EuScript{T}_\RR}
=\epi\varphi_\omega.
\end{align}
Consequently, $\varphi$ is normal.

\ref{t:3iv}:
It follows from \cite[Section~II.4.3]{Bour81} that
$(\YS\times\RR,\EuScript{T}_\YS\boxtimes\EuScript{T}_\RR)$ is a
separable Fr\'echet space.
Moreover, by \cite[Proposition~II.6.8]{Bour81}, $\XS\times\RR=
(\YS\times\RR,\EuScript{T}_\YS\boxtimes\EuScript{T}_\RR)^*$
and the weak topology of $\XS\times\RR$ is
$\EuScript{T}_\XS\boxtimes\EuScript{T}_\RR$. In turn, arguing as in
\cite[Section~IV-1.7]{Scha99}, we deduce that there exists
a covering
$(\boldsymbol{\mathsf{C}}_n)_{n\in\NN}$ of $\XS\times\RR$,
with respective
$\EuScript{T}_\XS\boxtimes\EuScript{T}_\RR$-induced topologies
$(\EuScript{T}_{\boldsymbol{\mathsf{C}}_n})_{n\in\NN}$,
such that, for every $n\in\NN$,
$(\boldsymbol{\mathsf{C}}_n,
\EuScript{T}_{\boldsymbol{\mathsf{C}}_n})$
is a compact separable metrizable space, hence a Polish space. We
also introduce
\begin{equation}
\label{e:q0x}
(\forall n\in\NN)\quad
Q_n\colon\Omega\times\boldsymbol{\mathsf{C}}_n\to\Omega\colon
(\omega,\mathsf{x},\xi)\mapsto\omega.
\end{equation}
Note that, for every subset $\boldsymbol{\mathsf{C}}$
of $\XS\times\RR$,
\begin{equation}
\label{e:rz2}
\menge{\omega\in\Omega}{\boldsymbol{\mathsf{C}}
\cap\epi\varphi_\omega\neq\emp}
=\bigcup_{n\in\NN}
\menge{\omega\in\Omega}{\boldsymbol{\mathsf{C}}\cap
\boldsymbol{\mathsf{C}}_n
\cap\epi\varphi_\omega\neq\emp}
=\bigcup_{n\in\NN}Q_n\Big(\boldsymbol{G}\cap
\big(\Omega\times(\boldsymbol{\mathsf{C}}\cap
\boldsymbol{\mathsf{C}}_n)\big)\Big).
\end{equation}

\ref{t:3iva}:
For every $n\in\NN$, set
\begin{equation}
\Omega_n=\menge{\omega\in\Omega}{
\boldsymbol{\mathsf{C}}_n\cap\epi\varphi_\omega\neq\emp},
\end{equation}
denote by $\FF_n$ the trace $\sigma$-algebra of $\FF$ on
$\Omega_n$, and observe that
\begin{equation}
\label{e:0f}
\Omega_n\in\FF\quad\text{and}\quad\FF_n\subset\FF.
\end{equation}
Now define
\begin{equation}
\label{e:kx}
\mathbb{K}=\menge{n\in\NN}{\Omega_n\neq\emp}
\quad\text{and}\quad
(\forall n\in\mathbb{K})\;\;
K_n\colon\Omega_n\to 2^{\boldsymbol{\mathsf{C}}_n}\colon
\omega\mapsto\boldsymbol{\mathsf{C}}_n\cap\epi\varphi_\omega.
\end{equation}
Then
\begin{equation}
\label{e:ky}
\mathbb{K}\neq\emp\quad\text{and}\quad
\bigcup_{n\in\mathbb{K}}\Omega_n=\Omega.
\end{equation}
Furthermore, the
$\EuScript{T}_\XS\boxtimes\EuScript{T}_\RR$-closedness of
$(\epi\varphi_\omega)_{\omega\in\Omega}$ guarantees that
\begin{equation}
(\forall n\in\mathbb{K})(\forall\omega\in\Omega)\quad
K_n(\omega)\;\text{is
$\EuScript{T}_{\boldsymbol{\mathsf{C}}_n}$-closed}.
\end{equation}
On the other hand, for every $n\in\mathbb{K}$ and every
closed subset
$\boldsymbol{\mathsf{D}}$ of $(\boldsymbol{\mathsf{C}}_n,
\EuScript{T}_{\boldsymbol{\mathsf{C}}_n})$,
there exists a closed subset
$\boldsymbol{\mathsf{E}}$ of
$(\XS\times\RR,\EuScript{T}_\XS\boxtimes\EuScript{T}_\RR)$
such that
$\boldsymbol{\mathsf{D}}=\boldsymbol{\mathsf{C}}_n\cap
\boldsymbol{\mathsf{E}}$ \cite[Section~I.3.1]{Bour71}
and therefore, since $\boldsymbol{\mathsf{C}}_n$ is
$\EuScript{T}_\XS\boxtimes\EuScript{T}_\RR$-closed,
we deduce from \eqref{e:0f} that
\begin{equation}
\menge{\omega\in\Omega_n}{\boldsymbol{\mathsf{D}}\cap K_n(\omega)
\neq\emp}
=\Omega_n\cap\menge{\omega\in\Omega}{
\boldsymbol{\mathsf{C}}_n\cap\boldsymbol{\mathsf{E}}
\cap\epi\varphi_\omega\neq\emp}
\in\FF_n.
\end{equation}
Hence, for every $n\in\mathbb{K}$, since
$(\boldsymbol{\mathsf{C}}_n,
\EuScript{T}_{\boldsymbol{\mathsf{C}}_n})$ is a Polish space,
we deduce from
\cite[Theorem~3.5(i), Theorem~5.1, and Theorem~5.6]{Himm75}
that there exist measurable mappings
$\boldsymbol{y}_n$ and $(\boldsymbol{z}_{n,k})_{k\in\NN}$ from
$(\Omega_n,\FF_n)$ to
$(\boldsymbol{\mathsf{C}}_n,\BE_{\boldsymbol{\mathsf{C}}_n})$
such that
\begin{equation}
\label{e:tp}
(\forall\omega\in\Omega_n)\quad
\boldsymbol{y}_n(\omega)\in K_n(\omega)\quad\text{and}\quad
K_n(\omega)
=\overline{\big\{\boldsymbol{z}_{n,k}(\omega)\big\}_{
k\in\NN}}^{\EuScript{T}_{\boldsymbol{\mathsf{C}}_n}}
=\boldsymbol{\mathsf{C}}_n\cap
\overline{\big\{\boldsymbol{z}_{n,k}(\omega)\big\}_{k\in\NN}}.
\end{equation}
In addition, since \cite[Theorem~3.5(i)]{Himm75} asserts that
\begin{align}
&
(\forall n\in\mathbb{K})\quad
\menge{(\omega,\mathsf{x},\xi)\in
\Omega_n\times\boldsymbol{\mathsf{C}}_n}{
(\mathsf{x},\xi)\in\boldsymbol{\mathsf{C}}_n\cap\epi\varphi_\omega}
\nonumber\\
&\hskip 26mm
=\menge{(\omega,\mathsf{x},\xi)\in
\Omega_n\times\boldsymbol{\mathsf{C}}_n}{
(\mathsf{x},\xi)\in K_n(\omega)}
\nonumber\\
&\hskip 26mm
\in\FF_n\otimes\BE_{\boldsymbol{\mathsf{C}}_n}
\nonumber\\
&\hskip 26mm
\subset\FF\otimes\BE_{\XS\times\RR},
\end{align}
we get from \eqref{e:s03} that
\begin{equation}
\boldsymbol{G}
=\bigcup_{n\in\mathbb{K}}\menge{(\omega,\mathsf{x},\xi)\in
\Omega_n\times\boldsymbol{\mathsf{C}}_n}{
(\mathsf{x},\xi)\in\boldsymbol{\mathsf{C}}_n\cap\epi\varphi_\omega}
\in\FF\otimes\BE_{\XS\times\RR}.
\end{equation}
Thus, in the light of \eqref{e:s04}, $\varphi$ is
$\FF\otimes\BE_\XS$-measurable.
Next, using \eqref{e:ky}, we construct a family
$(\Theta_n)_{n\in\mathbb{K}}$ of pairwise disjoint sets in $\FF$
such that
\begin{equation}
\label{e:yf}
\Theta_{\min\mathbb{K}}=\Omega_{\min\mathbb{K}},
\quad
\bigcup_{n\in\mathbb{K}}\Theta_n=\Omega,
\quad\text{and}\quad
(\forall n\in\mathbb{K})\;\;\Theta_n\subset\Omega_n.
\end{equation}
In turn, for every $\omega\in\Omega$, there exists a unique
$\ell_\omega\in\mathbb{K}$ such that
$\omega\in\Theta_{\ell_\omega}$. Therefore, appealing to
\eqref{e:yf}, the mapping
\begin{equation}
\label{e:8uy}
\boldsymbol{y}\colon\Omega\to\XS\times\RR\colon
\omega\mapsto\boldsymbol{y}_{\ell_\omega}(\omega)
\end{equation}
is well defined and, in view of \eqref{e:tp},
\begin{equation}
\label{e:z0}
(\forall\omega\in\Omega)\quad
\boldsymbol{y}(\omega)
=\boldsymbol{y}_{\ell_\omega}(\omega)
\in K_{\ell_\omega}(\omega)
\subset
\epi\varphi_\omega.
\end{equation}
Let $\boldsymbol{\mathsf{V}}\in
\EuScript{T}_\XS\boxtimes\EuScript{T}_\RR$.
Then, for every $n\in\mathbb{K}$,
$\boldsymbol{\mathsf{V}}\cap\boldsymbol{\mathsf{C}}_n$ is
$\EuScript{T}_{\boldsymbol{\mathsf{C}}_n}$-open
and thus the measurability of
$\boldsymbol{y}_n\colon(\Omega_n,\FF_n)\to
(\boldsymbol{\mathsf{C}}_n,\BE_{\boldsymbol{\mathsf{C}}_n})$ and
\eqref{e:0f} ensure that $\boldsymbol{y}_n^{-1}(
\boldsymbol{\mathsf{V}}\cap\boldsymbol{\mathsf{C}}_n)
\in\FF_n\subset\FF$. Hence, we infer from \eqref{e:yf},
\eqref{e:8uy}, and \eqref{e:tp} that
\begin{align}
\boldsymbol{y}^{-1}(\boldsymbol{\mathsf{V}})
&=\bigcup_{n\in\mathbb{K}}\menge{\omega\in\Theta_n}{
\boldsymbol{y}(\omega)\in\boldsymbol{\mathsf{V}}}
\nonumber\\
&=\bigcup_{n\in\mathbb{K}}\menge{\omega\in\Theta_n}{
\boldsymbol{y}_n(\omega)\in\boldsymbol{\mathsf{C}}_n\cap
\boldsymbol{\mathsf{V}}}
\nonumber\\
&=\bigcup_{n\in\mathbb{K}}\big(\Theta_n\cap
\boldsymbol{y}_n^{-1}(
\boldsymbol{\mathsf{C}}_n\cap\boldsymbol{\mathsf{V}})\big)
\nonumber\\
&\in\FF.
\end{align}
This verifies that $\boldsymbol{y}\colon(\Omega,\FF)\to
(\XS\times\RR,\BE_{\XS\times\RR})$ is measurable.
We now define
\begin{equation}
\label{e:xcd}
(\forall n\in\mathbb{K})(\forall k\in\NN)\quad
\boldsymbol{x}_{n,k}\colon\Omega\to\XS\times\RR\colon
\omega\mapsto
\begin{cases}
\boldsymbol{z}_{n,k}(\omega),&\text{if}\,\,\omega\in\Omega_n;\\
\boldsymbol{y}(\omega),
&\text{if}\,\,\omega\in\complement\Omega_n.
\end{cases}
\end{equation}
It results from \eqref{e:0f} that
$(\boldsymbol{x}_{n,k})_{n\in\mathbb{K},k\in\NN}$
are measurable mappings from
$(\Omega,\FF)$ to $(\XS\times\RR,\BE_{\XS\times\RR})$.
Furthermore, \eqref{e:tp} and \eqref{e:z0} give
\begin{equation}
\label{e:sd9}
(\forall n\in\mathbb{K})(\forall k\in\NN)
(\forall\omega\in\Omega)\quad
\boldsymbol{x}_{n,k}(\omega)\in\epi\varphi_\omega.
\end{equation}
Fix $\omega\in\Omega$ and let
$\boldsymbol{\mathsf{x}}\in\epi\varphi_\omega$.
Since $\bigcup_{n\in\mathbb{K}}(\boldsymbol{\mathsf{C}}_n\cap
\epi\varphi_\omega)=\epi\varphi_\omega$,
there exists $N\in\mathbb{K}$ such that
$\omega\in\Omega_N$ and
$\boldsymbol{\mathsf{x}}\in\boldsymbol{\mathsf{C}}_N
\cap\epi\varphi_\omega=K_N(\omega)$.
Thus, it results from \eqref{e:tp} and \eqref{e:xcd} that 
\begin{equation}
\boldsymbol{\mathsf{x}}
\in\overline{\big\{\boldsymbol{z}_{N,k}(\omega)\big\}_{k\in\NN}}
=\overline{\big\{\boldsymbol{x}_{N,k}(\omega)\big\}_{k\in\NN}}
\subset\overline{\big\{\boldsymbol{x}_{n,k}(\omega)\big\}_{
n\in\mathbb{K},k\in\NN}}.
\end{equation}
Therefore, since $\epi\varphi_\omega$ is closed,
it follows from \eqref{e:sd9} and \cite[Section~I.3.1]{Bour71}
that
\begin{equation}
\epi\varphi_\omega
=\overline{\big\{\boldsymbol{x}_{n,k}(\omega)\big\}_{
n\in\mathbb{K},k\in\NN}}.
\end{equation}
At the same time, for every $n\in\mathbb{K}$ and every $k\in\NN$,
since $\BE_{\XS\times\RR}=\BE_\XS\otimes\BE_\RR$
\cite[Lemma~6.4.2(i)]{Boga07} and since
$\boldsymbol{x}_{n,k}\colon(\Omega,\FF)\to
(\XS\times\RR,\BE_{\XS\times\RR})$ is measurable,
there exist $x_{n,k}\in\mathcal{L}(\Omega;\XS)$ and
$\varrho_{n,k}\in\mathcal{L}(\Omega;\RR)$ such that
$(\forall\omega\in\Omega)$
$\boldsymbol{x}_{n,k}(\omega)
=(x_{n,k}(\omega),\varrho_{n,k}(\omega))$.
Altogether, $\varphi$ is normal.

\ref{t:3ivb}$\Rightarrow$\ref{t:3iva}:
Let $\boldsymbol{\mathsf{C}}$ be a nonempty closed subset of
$(\XS\times\RR,\EuScript{T}_\XS\boxtimes\EuScript{T}_\RR)$.
Note that the lower semicontinuity of $\varphi$
ensures that $\boldsymbol{G}$ is closed.
For every $n\in\NN$, since
$\boldsymbol{G}\cap(\Omega\times(\boldsymbol{\mathsf{C}}\cap
\boldsymbol{\mathsf{C}}_n))$ is closed in
$(\Omega\times\boldsymbol{\mathsf{C}}_n,
\EuScript{T}_\Omega\boxtimes
\EuScript{T}_{\boldsymbol{\mathsf{C}}_n})$,
it follows from \eqref{e:q0x} and
\cite[Corollaire~I.10.5 and Th\'eor\`eme~I.10.1]{Bour71} that
$Q_n(\boldsymbol{G}\cap(\Omega\times(\boldsymbol{\mathsf{C}}\cap
\boldsymbol{\mathsf{C}}_n)))$ is closed in
$(\Omega,\EuScript{T}_\Omega)$
and, therefore, that it belongs to $\BE_\Omega=\FF$.
Thus, by
\eqref{e:rz2}, $\menge{\omega\in\Omega}{\boldsymbol{\mathsf{C}}
\cap\epi\varphi_\omega\neq\emp}\in\FF$.

\ref{t:3ivc}$\Rightarrow$\ref{t:3iva}:
There exists a topology $\widetilde{\EuScript{T}_\Omega}$ on
$\Omega$ such that
\begin{equation}
\EuScript{T}_\Omega\subset\widetilde{\EuScript{T}_\Omega}
\,\,\text{and}\,\,
\big(\Omega,\widetilde{\EuScript{T}_\Omega}\big)\,\,
\text{is a Polish space}.
\end{equation}
In addition, by \cite[Corollary~2, p.~101]{Schw73},
the Borel $\sigma$-algebra of
$(\Omega,\widetilde{\EuScript{T}_\Omega})$ is $\BE_\Omega=\FF$.
Let $\boldsymbol{\mathsf{C}}$ be a closed subset of
$(\XS\times\RR,\EuScript{T}_\XS\boxtimes\EuScript{T}_\RR)$ and
fix temporarily $n\in\NN$. Since the
$\FF\otimes\BE_\XS$-measurability of $\varphi$
and \eqref{e:s04} ensure that
$\boldsymbol{G}\in\FF\otimes\BE_{\XS\times\RR}$, we have
$\boldsymbol{G}\cap(\Omega\times(\boldsymbol{\mathsf{C}}\cap
\boldsymbol{\mathsf{C}}_n))
=\boldsymbol{G}\cap(\Omega\times\boldsymbol{\mathsf{C}})\cap
(\Omega\times\boldsymbol{\mathsf{C}}_n)
\in\BE_{\Omega\times\boldsymbol{\mathsf{C}}_n}$.
At the same time, for every $\omega\in\Omega$,
\begin{align}
&\menge{(\mathsf{x},\xi)\in\XS\times\RR}{
(\omega,\mathsf{x},\xi)\in
\boldsymbol{G}\cap\big(\Omega\times(\boldsymbol{\mathsf{C}}\cap
\boldsymbol{\mathsf{C}}_n)\big)}
\nonumber\\
&\hskip 26mm
=\menge{(\mathsf{x},\xi)\in\XS\times\RR}{(\mathsf{x},\xi)\in
\boldsymbol{\mathsf{C}}\cap\boldsymbol{\mathsf{C}}_n
\,\,\text{and}\,\,
(\mathsf{x},\xi)\in\epi\varphi_\omega},
\nonumber\\
&\hskip 26mm
=\boldsymbol{\mathsf{C}}\cap\boldsymbol{\mathsf{C}}_n
\cap\epi\varphi_\omega
\end{align}
is a closed subset of the compact space
$(\boldsymbol{\mathsf{C}}_n,
\EuScript{T}_{\boldsymbol{\mathsf{C}}_n})$.
In turn, since
$(\Omega,\widetilde{\EuScript{T}_\Omega})$ and
$(\boldsymbol{\mathsf{C}}_n,
\EuScript{T}_{\boldsymbol{\mathsf{C}}_n})$
are Polish spaces, \cite[Theorem~1]{Brow73} guarantees that
$Q_n(\boldsymbol{G}\cap(\Omega\times(\boldsymbol{\mathsf{C}}\cap
\boldsymbol{\mathsf{C}}_n)))\in\BE_\Omega=\FF$.
Consequently, we infer from \eqref{e:rz2} that
$\menge{\omega\in\Omega}{\boldsymbol{\mathsf{C}}
\cap\epi\varphi_\omega\neq\emp}\in\FF$.

\ref{t:3v}:
Let $(\YS,\EuScript{T}_\YS)$ be the strong dual of $\XS$.
Then $(\YS,\EuScript{T}_\YS)$ is a separable reflexive Banach
space. Consequently, \ref{t:3va} follows from \ref{t:3ivb}, and
\ref{t:3vb} follows from \ref{t:3ivc}.

\ref{t:3vi}$\Rightarrow$\ref{t:3vb}:
Let $\EuScript{T}_\Omega$ be the topology on $\Omega$ induced by
the standard topology on $\RR^M$. By
\cite[Corollary~1, p.~102]{Schw73},
$(\Omega,\EuScript{T}_\Omega)$ is a Lusin space.

\ref{t:3viia}:
The lower semicontinuity of $(\varphi_\omega)_{\omega\in\Omega}$
ensures that the sets $(\epi\varphi_\omega)_{\omega\in\Omega}$ are
closed. Hence, since $(\XS\times\RR,
\EuScript{T}_\XS\boxtimes\EuScript{T}_\RR)$ is a Polish
space, \cite[Theorem~3.5(i)]{Himm75} and \eqref{e:s03} yield
$\boldsymbol{G}\in\FF\otimes\BE_{\XS\times\RR}$. Therefore, by
\eqref{e:s04}, $\varphi$ is $\FF\otimes\BE_\XS$-measurable.
Consequently, we deduce the assertion from
\cite[Theorem~5.6]{Himm75}.

\ref{t:3viib}$\Rightarrow$\ref{t:3viia}:
This follows from \cite[Theorem~3.2(ii)]{Himm75}.

\ref{t:3viii}:
The $\BE_\XS$-measurability of $\mathsf{f}$ implies that
$\varphi$ is $\FF\otimes\BE_\XS$-measurable. At the same time,
since $(\XS\times\RR,
\EuScript{T}_\XS\boxtimes\EuScript{T}_\RR)$ is a Souslin space, we
deduce from \cite[Proposition~II.0]{Schw73} that there exists a
sequence
$\{(\mathsf{x}_n,\xi_n)\}_{n\in\NN}$ in $\epi\mathsf{f}$ such that
$\overline{\{(\mathsf{x}_n,\xi_n)\}_{n\in\NN}}
=\overline{\epi\mathsf{f}}$. Altogether, upon setting
\begin{equation}
(\forall n\in\NN)\quad
x_n\colon\Omega\to\XS\colon\omega\mapsto\mathsf{x}_n
\quad\text{and}\quad
\varrho_n\colon\Omega\to\RR\colon\omega\mapsto\xi_n,
\end{equation}
we conclude that $\varphi$ is normal.
\end{proof}

\begin{remark}
\label{r:7}
Here are a few observations about Definition~\ref{d:n}.
\begin{enumerate}
\item
\label{r:7i}
The setting of Theorem~\ref{t:3}\ref{t:3viib} corresponds to
the definition of normality in \cite{Roc76k}.
\item
\label{r:7ii}
The setting of Theorem~\ref{t:3}\ref{t:3ia} corresponds to
the definition of normality in \cite{Vala75}, which itself
contains that of \cite{Rock71}.
\item
\label{r:7iii}
The frameworks of \ref{r:7i} and \ref{r:7ii} above are distinct since the
former does not require that $(\Omega,\FF,\mu)$ be complete.
Definition~\ref{d:n} unifies them and, as seen in Theorem~\ref{t:3},
goes beyond. For the importance of noncompleteness in applications, see
for instance \cite{Rans90} and \cite[p.~649]{Rock09}.
\end{enumerate}
\end{remark}

\section{Interchange rules with compliant spaces and normal
integrands}
\label{sec:5}

The main result of this section is the following interchange
theorem, which brings together the abstract principle of
Theorem~\ref{t:1}, the notion of compliance of
Definition~\ref{d:1}, and the notion of normality of
Definition~\ref{d:n}.

\begin{theorem}
\label{t:8}
Suppose that Assumption~\ref{a:1} holds, that $\XX$ is compliant,
and that $\varphi$ is normal. Then 
\begin{equation}
\label{e:311}
\inf_{x\in\XX}\int_\Omega\varphi\big(\omega,x(\omega)\big)
\mu(d\omega)=
\int_\Omega\inf_{\mathsf{x}\in\XS}\varphi(\omega,\mathsf{x})\,
\mu(d\omega).
\end{equation}
\end{theorem}
\begin{proof}
We apply Theorem~\ref{t:1}.
By virtue of the normality of $\varphi$, per Definition~\ref{d:n},
we choose sequences $(z_n)_{n\in\NN}$ in
$\mathcal{L}(\Omega;\XS)$ and $(\vartheta_n)_{n\in\NN}$ in
$\mathcal{L}(\Omega;\RR)$ such that
\begin{equation}
\label{e:n0d}
(\forall\omega\in\Omega)\quad
\big\{\big(z_n(\omega),\vartheta_n(\omega)\big)\big\}_{n\in\NN}
\subset\epi\varphi_\omega
\quad\text{and}\quad
\overline{\epi\varphi_\omega}=\overline{\big\{\big(z_n(\omega),
\vartheta_n(\omega)\big)\big\}_{n\in\NN}}.
\end{equation}
On the other hand, Assumption~\ref{a:1}\ref{a:1f} ensures that
$(\forall\omega\in\Omega)$ $\inf\varphi(\omega,\XS)<\pinf$.
Now fix $\omega\in\Omega$ and let
$\xi\in\left]\inf\varphi(\omega,\XS),\pinf\right[$.
Then there exits $\mathsf{x}\in\XS$ such that
$(\mathsf{x},\xi)\in\epi\varphi_\omega$. Thus,
in view of \eqref{e:n0d}, we obtain a subnet
$(\vartheta_{k(b)}(\omega))_{b\in B}$ of
$(\vartheta_n(\omega))_{n\in\NN}$ such that
$\vartheta_{k(b)}(\omega)\to\xi$. On the other hand,
\begin{equation}
(\forall b\in B)\quad
\inf\varphi(\omega,\XS)
\leq\inf_{n\in\NN}\varphi\big(\omega,z_n(\omega)\big)
\leq\varphi\big(\omega,z_{k(b)}(\omega)\big)
\leq\vartheta_{k(b)}(\omega).
\end{equation}
Hence $\inf\varphi(\omega,\XS)
\leq\inf_{n\in\NN}\varphi(\omega,z_n(\omega))
\leq\xi$. In turn,
letting $\xi\downarrow\inf\varphi(\omega,\XS)$ yields
$\inf\varphi(\omega,\XS)=
\inf_{n\in\NN}\varphi(\omega,z_n(\omega))$.
Therefore, property~\ref{t:1iia} in Theorem~\ref{t:1}
is satisfied with $(\forall n\in\NN)$ $x_n=z_n-\overline{x}$.
At the same time, property~\ref{t:1iib} in Theorem~\ref{t:1}
follows from Assumption~\ref{a:1}\ref{a:1d} and the compliance of
$\XX$. Finally, since the functions
$(\varphi(\Cdot,z_n(\Cdot)))_{n\in\NN}$ are $\FF$-measurable by
Assumption~\ref{a:1}\ref{a:1f}, so is
$\inf_{n\in\NN}\varphi(\Cdot,z_n(\Cdot))=\inf\varphi(\Cdot,\XS)$.
\end{proof}

In the remainder of this section, we construct new scenarios for
the validity of the interchange rule as instantiations of
Theorem~\ref{t:8}.

\begin{example}
\label{ex:1}
Let $\XS$ be a separable real Banach space with strong topology
$\EuScript{T}_\XS$, let $(\Omega,\FF,\mu)$ be a $\sigma$-finite
measure space such that $\mu(\Omega)\neq 0$, let $\XX$ be a vector
subspace of $\mathcal{L}(\Omega;\XS)$, and let
$\varphi\colon(\Omega\times\XS,\FF\otimes\BE_\XS)\to\RXX$ be
measurable. Suppose that the following are satisfied:
\begin{enumerate}
\item
\label{ex:1i}
For every $A\in\FF$ such that $\mu(A)<\pinf$ and every
$z\in\mathcal{L}^\infty(\Omega;\XS)$, $1_Az\in\XX$.
\item
\label{ex:1ii}
$\varphi$ is normal.
\item
\label{ex:1iii}
There exists $\overline{x}\in\XX$ such that
$\int_\Omega\max\{\varphi(\Cdot,\overline{x}(\Cdot)),0\}
d\mu<\pinf$.
\end{enumerate}
Then the interchange rule \eqref{e:311} holds.
\end{example}
\begin{proof}
Note that Assumption~\ref{a:1} is satisfied. Hence, the assertion
follows from Proposition~\ref{p:10}\ref{p:10ii} and
Theorem~\ref{t:8}.
\end{proof}

\begin{example}
\label{ex:2}
Suppose that Assumption~\ref{a:1} holds, that $(\Omega,\FF,\mu)$
is complete, and that $\XX$ is compliant. Then the interchange
rule \eqref{e:311} holds.
\end{example}
\begin{proof}
Combine Theorem~\ref{t:3}\ref{t:3ia} and Theorem~\ref{t:8}.
\end{proof}

When specialized to probability in separable Banach spaces,
Theorem~\ref{t:8} yields conditions for the interchange of
infimization and expectation. Here is an illustration.

\begin{example}
\label{ex:4}
Let $\XS$ be a separable real Banach space,
let $(\Omega,\FF,\PP)$ be a probability space,
let $\XX$ be a vector subspace of
$\mathcal{L}(\Omega;\XS)$ which contains
$\mathcal{L}^\infty(\Omega;\XS)$,
and let $\varphi\colon(\Omega\times\XS,\FF\otimes\BE_\XS)\to\RXX$
be normal. In addition, set $\phi=\inf\varphi(\Cdot,\XS)$ and
$\Phi\colon\mathcal{L}(\Omega;\XS)
\to\mathcal{L}(\Omega;\RXX)\colon
x\mapsto\varphi(\Cdot,x(\Cdot))$, and suppose that there
exists $\overline{x}\in\XX$ such that
$\EE\max\{\Phi(\overline{x}),0\}<\pinf$. Then
\begin{equation}
\inf_{x\in\XX}\EE\Phi(x)=\EE\phi.
\end{equation}
\end{example}
\begin{proof}
This is a special case of Example~\ref{ex:1}.
\end{proof}

\begin{example}
\label{ex:9}
Suppose that Assumption~\ref{a:1} holds, that $\XX$ is compliant,
and that the functions $(\varphi_\omega)_{\omega\in\Omega}$ are
upper semicontinuous. Then the interchange rule \eqref{e:311}
holds.
\end{example}
\begin{proof}
We deduce from Assumption~\ref{a:1}\ref{a:1f} and
Theorem~\ref{t:3}\ref{t:3id} that $\varphi$ is normal. Thus, the
conclusion follows from Theorem~\ref{t:8}.
\end{proof}

An important realization of Example~\ref{ex:9} is the case of
Carath\'eodory integrands.

\begin{example}[Carath\'eodory integrand]
\label{ex:5}
Let $(\XS,\EuScript{T}_\XS)$ be a Souslin topological vector
space, let $(\Omega,\FF,\mu)$ be a $\sigma$-finite measure space
such that $\mu(\Omega)\neq 0$, let $\XX$ be a compliant vector
subspace of $\mathcal{L}(\Omega;\XS)$, and
let $\varphi\colon\Omega\times\XS\to\RXX$ be
a Carath\'eodory integrand in the sense that,
for every $(\omega,\mathsf{x})\in\Omega\times\XS$,
$\varphi(\omega,\Cdot)$ is continuous with
$\epi\varphi_\omega\neq\emp$, and
$\varphi(\Cdot,\mathsf{x})$ is $\FF$-measurable. Suppose that
there exists $\overline{x}\in\XX$
such that $\int_\Omega\max\{\varphi(\Cdot,\overline{x}(\Cdot)),0\}
d\mu<\pinf$. Then the interchange rule \eqref{e:311} holds.
\end{example}
\begin{proof}
Since $(\XS,\EuScript{T}_\XS)$ is a Souslin topological vector
space, \cite[Section~35F,~p.~244]{Will70} implies that
it is a regular Souslin space. Thus, we
deduce from Theorem~\ref{t:3}\ref{t:3iii} that $\varphi$ is
normal and, in particular, it is $\FF\otimes\BE_\XS$-measurable.
Hence, Assumption~\ref{a:1} is satisfied. Consequently,
Example~\ref{ex:9} yields the conclusion.
\end{proof}

\begin{remark}
\label{r:3}
Here are connections with existing work.
\begin{enumerate}
\item
\label{r:3i}
Example~\ref{ex:1} unifies and extends the classical results of
\cite{Hiai77,Rock71,Roc76k}:
\begin{itemize}
\item
It captures \cite[Theorem~3A]{Roc76k}, where $\XS$ is a Euclidean
space and $\XX$ is assumed to be Rockafellar-decomposable
(see Proposition~\ref{p:10}\ref{p:10iv} for definition).
\item
It covers the setting of \cite{Rock71}, where
$(\Omega,\FF\,\mu)$ is assumed to be complete
and where \ref{ex:1i} and \ref{ex:1ii} in Example~\ref{ex:1}
are specialized to:
\begin{enumerate}[label={\rm(\roman*')}]
\setcounter{enumi}{1}
\item
\label{r:3i+}
$\XX$ is Rockafellar-decomposable.
\item
\label{r:3ii+}
The functions $(\varphi_\omega)_{\omega\in\Omega}$ are lower
semicontinuous.
\end{enumerate}
The fact that property~\ref{ex:1ii} in Example~\ref{ex:1} is
satisfied when $(\Omega,\FF,\mu)$ is complete is shown in
Theorem~\ref{t:3}\ref{t:3ia}.
\item
It captures \cite[Theorem~2.2]{Hiai77}, where
$\XX=\menge{x\in\mathcal{L}(\Omega;\XS)}{
\int_\Omega\|x(\omega)\|_\XS^p\,\mu(d\omega)<\pinf}$
with $p\in\left[1,\pinf\right[$.
\end{itemize}
\item
An important contribution of Theorem~\ref{t:8} and, in particular,
of Example~\ref{ex:1} is that completeness of the measure space
$(\Omega,\FF,\mu)$ is not required.
\item
In the special case when $\XS$ is a
Banach space, an alternative framework that recovers the
interchange rules of \cite{Hiai77,Rock71,Roc76k} was proposed in
\cite[Theorem~6.1]{Gine09}, where the right-hand side of
\eqref{e:1} is replaced by the integral of an abstract essential
infimum. However, \cite{Gine09} does not provide new scenarios for
\eqref{e:1} beyond the known cases in Banach spaces.
An interpretation of the framework of \cite{Gine09}
from the view point of monotone relations between partially
ordered sets is proposed in \cite{Chan22}.
\item
Example~\ref{ex:2} captures \cite[Theorem~4]{Perk18}, where
$\mu(\Omega)<\pinf$ and $\XX$ is Valadier-decomposable
(see Proposition~\ref{p:10}\ref{p:10v} for definition).
It also covers the setting of \cite{Vala75}, where $\XS$ is a
Souslin topological vector space and $\XX$ is
Valadier-decomposable.
\item
Example~\ref{ex:4} contains the interchange rule of
\cite{Penn23,Shap21}, where $\XS$ is the standard Euclidean space
$\RR^N$ and $\XX$ is Rockafellar-decomposable.
\item
Example~\ref{ex:5} extends
\cite[Theorem~3A]{Roc76k}, where $\XS$ is the standard Euclidean
space $\RR^N$ and $\XX$ is Rockafellar-decomposable.
\end{enumerate}
\end{remark}

\section{Interchanging convex-analytical operations and
integration}
\label{sec:6}

We put the interchange principle of Theorem~\ref{t:1}, compliance,
and normality in action to evaluate convex-analytical objects
associated with integral functions, namely conjugate functions,
subdifferential operators, recession functions, Moreau envelopes,
and proximity operators. This analysis results in new interchange
rules for the convex calculus of integral functions. Throughout
this section, we adopt the following notation.

\begin{notation}
\label{n:1}
Let $(\XS,\EuScript{T}_\XS)$ be a real topological vector space,
let $(\Omega,\FF,\mu)$ be a $\sigma$-finite measure space such
that $\mu(\Omega)\neq 0$, let $\XX$ be a vector subspace of
$\mathcal{L}(\Omega;\XS)$, and let
$\varphi\colon(\Omega\times\XS,\FF\otimes\BE_\XS)\to\RXX$ be an
integrand. Then:
\begin{enumerate}
\item
$\TXX$ is the vector space of equivalence classes of $\mae$ equal
mappings in $\XX$.
\item
The equivalence class in $\TXX$ of $x\in\XX$ is denoted by
$\widetilde{x}$. Conversely, an arbitrary representative in $\XX$
of $\widetilde{x}\in\TXX$ is denoted by $x$.
\item
$\mathfrak{I}_{\varphi,\TXX}\colon\TXX\to\RXX\colon
\widetilde{x}\mapsto
\int_\Omega\varphi(\omega,x(\omega))\mu(d\omega)$.
\end{enumerate}
\end{notation}

We shall require the following result. Its item \ref{l:6i} 
appears in \cite[Lemma~4]{Vala75} in the special case when 
$(\Omega,\FF,\mu)$ is complete.

\begin{lemma}
\label{l:6}
Let $(\Omega,\FF,\mu)$ be a $\sigma$-finite measure space
such that $\mu(\Omega)\neq 0$,
let $(\XS,\EuScript{T}_\XS)$ be a Souslin locally convex real
topological vector space, and let $(\YS,\EuScript{T}_\YS)$ be a
separable locally convex real topological vector space.
Suppose that $\XS$ and $\YS$ are placed in separating duality via
a bilinear form
$\pair{\Cdot}{\Cdot}_{\XS,\YS}\colon\XS\times\YS\to\RR$ with which
$\EuScript{T}_\XS$ and $\EuScript{T}_\YS$ are compatible.
Then the following hold:
\begin{enumerate}
\item
\label{l:6i}
$\pair{\Cdot}{\Cdot}_{\XS,\YS}\colon(\XS\times\YS,
\BE_\XS\otimes\BE_\YS)\to\RR$ is measurable.
\item
\label{l:6ii}
Let $\XX\subset\mathcal{L}(\Omega;\XS)$ and
$\YY\subset\mathcal{L}(\Omega;\YS)$ be vector subspaces
such that the following are satisfied:
\begin{enumerate}
\item
\label{l:6iia}
$(\forall x\in\XX)(\forall y\in\YY)$
$\int_\Omega|\pair{x(\omega)}{y(\omega)}_{\XS,\YS}|
\mu(d\omega)<\pinf$.
\item
\label{l:6iib}
$\bigcup_{\mathsf{x}\in\XS}\menge{1_A\mathsf{x}}{A\in\FF\,\,
\text{and}\,\,\mu(A)<\pinf}\subset\XX$.
\item
\label{l:6iic}
$\bigcup_{\mathsf{y}\in\YS}\menge{1_A\mathsf{y}}{A\in\FF\,\,
\text{and}\,\,\mu(A)<\pinf}\subset\YY$.
\end{enumerate}
Then $\widetilde{\XX}$ and $\widetilde{\YY}$ are in
separating duality via the bilinear form $\pair{\Cdot}{\Cdot}$
defined by
\begin{equation}
\label{e:62}
(\forall\widetilde{x}\in\TXX)(\forall\widetilde{y}\in\TYY)\quad
\pair{\widetilde{x}}{\widetilde{y}}=
\int_\Omega\Pair{x(\omega)}{y(\omega)}_{\XS,\YS}\mu(d\omega).
\end{equation}
\end{enumerate}
\end{lemma}
\begin{proof}
\ref{l:6i}:
We deduce from \cite[Section~35F,~p.~244]{Will70} that
$(\XS,\EuScript{T}_\XS)$ is a regular Souslin space.
On the other hand, since $\EuScript{T}_\YS$ and $\EuScript{T}_\XS$
are compatible with $\pair{\Cdot}{\Cdot}_{\XS,\YS}$,
the functions
$(\pair{\mathsf{x}}{\Cdot}_{\XS,\YS})_{\mathsf{x}\in\XS}$ are
$\BE_\YS$-measurable and the functions
$(\pair{\Cdot}{\mathsf{y}}_{\XS,\YS})_{\mathsf{y}\in\YS}$ are
continuous. Hence, Theorem~\ref{t:3}\ref{t:3iii}
implies that $\pair{\Cdot}{\Cdot}_{\XS,\YS}\colon(\XS\times\YS,
\BE_\XS\otimes\BE_\YS)\to\RR$ is measurable.

\ref{l:6ii}:
Note that \ref{l:6i} guarantees that, for every $x\in\XX$ and every
$y\in\YY$, $\pair{x(\Cdot)}{y(\Cdot)}_{\XS,\YS}$ is
$\FF$-measurable. Now let $\{\mathsf{y}_n\}_{n\in\NN}$ be a
dense subset of $(\YS,\EuScript{T}_\YS)$ and let
$\widetilde{x}\in\TXX$ be such
that $(\forall\widetilde{y}\in\TYY)$
$\pair{\widetilde{x}}{\widetilde{y}}=0$.
Then, for every $n\in\NN$ and every $A\in\FF$ such that
$\mu(A)<\pinf$, since \ref{l:6iic} ensures that
$1_A\mathsf{y}_n\in\YY$, we deduce from \eqref{e:62} that
$\int_A\pair{x(\omega)}{\mathsf{y}_n}_{\XS,\YS}\mu(d\omega)
=\int_\Omega\pair{x(\omega)}{1_A(\omega)\mathsf{y}_n}_{\XS,\YS}
\mu(d\omega)=0$. Therefore, since $(\Omega,\FF,\mu)$ is
$\sigma$-finite, it follows that $(\forall n\in\NN)$
$\pair{x(\Cdot)}{\mathsf{y}_n}_{\XS,\YS}=0$ $\mae$
Thus $\widetilde{x}=0$. Likewise,
$(\forall\widetilde{y}\in\TYY)$ $\pair{\Cdot}{\widetilde{y}}=0$
$\Rightarrow$ $\widetilde{y}=0$, which completes the proof.
\end{proof}

The main result of this section is set in the following
environment, which is well defined by virtue of Lemma~\ref{l:6}.

\begin{assumption}
\label{a:2}
\
\begin{enumerate}[label={\rm[\Alph*]}]
\item
\label{a:2a}
$(\XS,\EuScript{T}_\XS)$ is a Souslin locally convex real
topological vector space and $(\YS,\EuScript{T}_\YS)$ is a
separable locally convex real topological vector space.
In addition, $\XS$ and $\YS$ are placed in separating duality via
a bilinear form
$\pair{\Cdot}{\Cdot}_{\XS,\YS}\colon\XS\times\YS\to\RR$ with which
$\EuScript{T}_\XS$ and $\EuScript{T}_\YS$ are compatible.
\item
\label{a:2b}
$(\Omega,\FF,\mu)$ is a $\sigma$-finite measure space such that
$\mu(\Omega)\neq 0$.
\item
\label{a:2c}
$\XX\subset\mathcal{L}(\Omega;\XS)$ and
$\YY\subset\mathcal{L}(\Omega;\YS)$ are vector subspaces
such that $(\forall x\in\XX)(\forall y\in\YY)$
$\int_\Omega|\pair{x(\omega)}{y(\omega)}_{\XS,\YS}|
\mu(d\omega)<\pinf$. In addition,
\begin{equation}
\label{e:k7}
\XX\,\,\text{is compliant and}\,\,
\bigcup_{\mathsf{y}\in\YS}\menge{1_A\mathsf{y}}{A\in\FF\,\,
\text{and}\,\,\mu(A)<\pinf}\subset\YY.
\end{equation}
\item
\label{a:2d}
$\widetilde{\XX}$ and $\widetilde{\YY}$ are placed in separating
duality via the bilinear form $\pair{\Cdot}{\Cdot}$ defined by
\begin{equation}
\label{e:63}
(\forall\widetilde{x}\in\TXX)
(\forall\widetilde{y}\in\TYY)\quad
\pair{\widetilde{x}}{\widetilde{y}}=
\int_\Omega\Pair{x(\omega)}{y(\omega)}_{\XS,\YS}\mu(d\omega),
\end{equation}
and they are equipped with locally convex Hausdorff topologies
which are compatible with $\pair{\Cdot}{\Cdot}$.
\item
\label{a:2e}
$\varphi\colon(\Omega\times\XS,\FF\otimes\BE_\XS)\to\RX$ is normal
and we write
$\varphi^*\colon\Omega\times\YS\to\RXX\colon
(\omega,\mathsf{y})\mapsto\varphi_\omega^*(\mathsf{y})$.
\item
\label{a:2f}
$\dom\mathfrak{I}_{\varphi,\TXX}\neq\emp$.
\end{enumerate}
\end{assumption}

\begin{proposition}
\label{p:6}
Suppose that Assumption~\ref{a:2} holds. Then
$\varphi^*$ is $\FF\otimes\BE_\YS$-measurable.
\end{proposition}
\begin{proof}
According to Assumption~\ref{a:2}\ref{a:2e} and
Definition~\ref{d:n}, there exist sequences $(x_n)_{n\in\NN}$ in
$\mathcal{L}(\Omega;\XS)$ and $(\varrho_n)_{n\in\NN}$ in
$\mathcal{L}(\Omega;\RR)$ such that
\begin{equation}
\label{e:ef}
(\forall\omega\in\Omega)\quad
\big\{\big(x_n(\omega),\varrho_n(\omega)\big)\big\}_{n\in\NN}
\subset\epi\varphi_\omega
\quad\text{and}\quad
\overline{\epi\varphi_\omega}=\overline{\big\{\big(x_n(\omega),
\varrho_n(\omega)\big)\big\}_{n\in\NN}}.
\end{equation}
Set
\begin{equation}
(\forall n\in\NN)\quad
\psi_n\colon\Omega\times\YS\to\RR\colon(\omega,\mathsf{y})\mapsto
\pair{x_n(\omega)}{\mathsf{y}}_{\XS,\YS}-\varrho_n(\omega).
\end{equation}
Then, for every $n\in\NN$, 
Assumption~\ref{a:2}\ref{a:2a}--\ref{a:2c}
and Lemma~\ref{l:6}\ref{l:6i} ensure that
$\psi_n$ is $\FF\otimes\BE_\YS$-measurable. On the other hand,
since the functions
$(\pair{\Cdot}{\mathsf{y}}_{\XS,\YS})_{\mathsf{y}\in\YS}$ are
continuous, we derive from Assumption~\ref{a:2}\ref{a:2e},
\eqref{e:l0d}, and \eqref{e:ef} that
\begin{align}
\big(\forall(\omega,\mathsf{y})\in\Omega\times\YS\big)\quad
\varphi^*(\omega,\mathsf{y})
&=\sup_{(\mathsf{x},\xi)\in\epi\varphi_\omega}\big(
\pair{\mathsf{x}}{\mathsf{y}}_{\XS,\YS}-\xi\big)
\nonumber\\
&=\sup_{(\mathsf{x},\xi)\in\overline{\epi\varphi_\omega}}\big(
\pair{\mathsf{x}}{\mathsf{y}}_{\XS,\YS}-\xi\big)
\nonumber\\
&=\sup_{n\in\NN}\big(\pair{x_n(\omega)}{\mathsf{y}}_{\XS,\YS}
-\varrho_n(\omega)\big)
\nonumber\\
&=\sup_{n\in\NN}\psi_n(\omega,\mathsf{y}).
\end{align}
Thus $\varphi^*$ is $\FF\otimes\BE_\YS$-measurable.
\end{proof}

We first investigate the conjugate and the subdifferential
of integral functions.

\begin{theorem}
\label{t:2}
Suppose that Assumption~\ref{a:2} holds.
Then the following are satisfied:
\begin{enumerate}
\item
\label{t:2i}
$\mathfrak{I}_{\varphi,\TXX}^*=\mathfrak{I}_{\varphi^*,\TYY}$.
\item
\label{t:2ii}
Suppose that $\mathfrak{I}_{\varphi,\TXX}$ is proper,
let $\widetilde{x}\in\TXX$, and let $\widetilde{y}\in\TYY$. Then
$\widetilde{y}\in
\partial\mathfrak{I}_{\varphi,\TXX}(\widetilde{x})$
$\Leftrightarrow$ $y(\omega)\in\partial\varphi_\omega(x(\omega))$
for $\mu$-almost every $\omega\in\Omega$.
\end{enumerate}
\end{theorem}
\begin{proof}
\ref{t:2i}:
In view of Assumption~\ref{a:2}\ref{a:2e} and
Proposition~\ref{p:6}, $\mathfrak{I}_{\varphi,\TXX}$
and $\mathfrak{I}_{\varphi^*,\TYY}$ are well defined.
Further, there exist sequences $(z_n)_{n\in\NN}$ in
$\mathcal{L}(\Omega;\XS)$ and $(\vartheta_n)_{n\in\NN}$ in
$\mathcal{L}(\Omega;\RR)$ such that
\begin{equation}
\label{e:imd}
(\forall\omega\in\Omega)\quad
\big\{\big(z_n(\omega),\vartheta_n(\omega)\big)\big\}_{n\in\NN}
\subset\epi\varphi_\omega
\quad\text{and}\quad
\overline{\epi\varphi_\omega}=\overline{\big\{\big(z_n(\omega),
\vartheta_n(\omega)\big)\big\}_{n\in\NN}}.
\end{equation}
Let $\widetilde{y}\in\TYY$, define
$\psi\colon\Omega\times\XS\to\RX\colon
(\omega,\mathsf{x})\mapsto\varphi_\omega(\mathsf{x})
-\pair{\mathsf{x}}{y(\omega)}_{\XS,\YS}$, and note that
$(\forall\omega\in\Omega)$ $\epi\psi_\omega\neq\emp$.
Assumption~\ref{a:2}\ref{a:2e} and
Lemma~\ref{l:6}\ref{l:6i} imply that
\begin{equation}
\label{e:tys}
\psi\,\,\text{is $\FF\otimes\BE_\XS$-measurable}.
\end{equation}
Moreover, using the continuity of the linear functionals
$(\pair{\Cdot}{\mathsf{y}}_{\XS,\YS})_{\mathsf{y}\in\YS}$,
we derive from \eqref{e:imd} that
\begin{align}
(\forall\omega\in\Omega)\quad
\inf\psi(\omega,\XS)
&=\inf_{(\mathsf{x},\xi)\in\epi\varphi_\omega}\big(
\xi-\pair{\mathsf{x}}{y(\omega)}_{\XS,\YS}\big)
\nonumber\\
&=\inf_{(\mathsf{x},\xi)\in\overline{\epi\varphi_\omega}}\big(
\xi-\pair{\mathsf{x}}{y(\omega)}_{\XS,\YS}\big)
\nonumber\\
&=\inf_{n\in\NN}\big(\vartheta_n(\omega)
-\pair{z_n(\omega)}{y(\omega)}\big)
\nonumber\\
&\geq\inf_{n\in\NN}\big(\varphi_\omega\big(z_n(\omega)\big)
-\pair{z_n(\omega)}{y(\omega)}\big)
\nonumber\\
&=\inf_{n\in\NN}\psi\big(\omega,z_n(\omega)\big)
\nonumber\\
&\geq\inf\psi(\omega,\XS).
\end{align}
Hence, $(\forall\omega\in\Omega)$
$\inf\psi(\omega,\XS)=\inf_{n\in\NN}\psi(\omega,z_n(\omega))$.
Combining this with \eqref{e:tys},
we infer that $\inf\psi(\Cdot,\XS)$ is $\FF$-measurable and that
$\psi$ fulfills property~\ref{t:1iia} in Theorem~\ref{t:1} with
$(\forall n\in\NN)$ $x_n=z_n-\overline{x}$.
In turn, thanks to Assumption~\ref{a:2}\ref{a:2b}
and the compliance of $\XX$, property~\ref{t:1iib} in 
Theorem~\ref{t:1} is fulfilled.
Thus, by invoking \eqref{e:63} and Theorem~\ref{t:1}, we obtain
\begin{align}
\mathfrak{I}_{\varphi,\TXX}^*(\widetilde{y})
&=\sup_{\widetilde{x}\in\TXX}\big(
\pair{\widetilde{x}}{\widetilde{y}}-
\mathfrak{I}_{\varphi,\TXX}(\widetilde{x})\big)
\nonumber\\
&=\sup_{x\in\XX}\bigg(
\int_\Omega\Pair{x(\omega)}{y(\omega)}_{\XS,\YS}\mu(d\omega)
-\int_\Omega\varphi\big(\omega,x(\omega)\big)\mu(d\omega)
\bigg)
\nonumber\\
&=-\inf_{x\in\XX}\int_\Omega\psi\big(\omega,x(\omega)\big)
\mu(d\omega)
\nonumber\\
&=-\int_\Omega\inf_{\mathsf{x}\in\XS}\psi(\omega,\mathsf{x})
\,\mu(d\omega)
\nonumber\\
&=\int_\Omega\varphi_\omega^*\big(y(\omega)\big)\mu(d\omega),
\end{align}
as desired.

\ref{t:2ii}:
Since the functions $(\varphi_\omega)_{\omega\in\Omega}$
are proper by Assumption~\ref{a:2}\ref{a:2e}, we derive from
\eqref{e:s14}, \ref{t:2i}, \eqref{e:63}, and the
Fenchel--Young inequality that
\begin{align}
\widetilde{y}\in
\partial\mathfrak{I}_{\varphi,\TXX}(\widetilde{x})
&\Leftrightarrow
\mathfrak{I}_{\varphi,\TXX}(\widetilde{x})+
\mathfrak{I}_{\varphi^*,\TYY}(\widetilde{y})
=\pair{\widetilde{x}}{\widetilde{y}}
\nonumber\\
&\Leftrightarrow
\int_\Omega
\varphi_\omega\big(x(\omega)\big)\mu(d\omega)
+\int_\Omega\varphi_\omega^*\big(y(\omega)\big)\mu(d\omega)
=\int_\Omega\Pair{x(\omega)}{y(\omega)}_{\XS,\YS}\mu(d\omega)
\nonumber\\
&\Leftrightarrow
\varphi_\omega\big(x(\omega)\big)
+\varphi_\omega^*\big(y(\omega)\big)
=\Pair{x(\omega)}{y(\omega)}_{\XS,\YS}\,\,\mae
\nonumber\\
&\Leftrightarrow
y(\omega)\in\partial\varphi_\omega\big(x(\omega)\big)\,\,\mae,
\end{align}
which completes the proof.
\end{proof}

A first important consequence of Theorem~\ref{t:2}\ref{t:2i} is the
following.

\begin{proposition}
\label{p:8}
Suppose that Assumption~\ref{a:2} holds,
that $(\YS,\EuScript{T}_\YS)$ is a Souslin space,
that $\dom\mathfrak{I}_{\varphi^*,\TYY}\neq\emp$,
that $\YY$ is compliant, and that
$(\forall\omega\in\Omega)$ $\varphi_\omega\in\Gamma_0(\XS)$.
Then the following are satisfied:
\begin{enumerate}
\item
\label{p:8i}
$\mathfrak{I}_{\varphi,\TXX}\in\Gamma_0(\TXX)$.
\item
\label{p:8ii}
Set $\rec\varphi\colon\Omega\times\XS\to\RX\colon
(\omega,\mathsf{x})\mapsto(\rec\varphi_\omega)(\mathsf{x})$.
Then $\rec\varphi$ is $\FF\otimes\BE_\XS$-measurable and 
$\rec\mathfrak{I}_{\varphi,\TXX}=\mathfrak{I}_{\rec\varphi,\TXX}$.
\end{enumerate}
\end{proposition}
\begin{proof}
\ref{p:8i}:
Let $\widetilde{x}\in\TXX$ and set
\begin{equation}
\label{e:r3}
\psi\colon\Omega\times\YS\to\RX\colon(\omega,\mathsf{y})
\mapsto\varphi_\omega^*(\mathsf{y})
-\pair{x(\omega)}{\mathsf{y}}_{\XS,\YS}
\quad\text{and}\quad
\vartheta=\inf\psi(\Cdot,\YS).
\end{equation}
By Assumption~\ref{a:2}\ref{a:2e},
\begin{equation}
\label{e:fxs}
\varphi\big(\Cdot,x(\Cdot)\big)\,\,\text{is $\FF$-measurable},
\end{equation}
while it results from Proposition~\ref{p:6} and
Lemma~\ref{l:6}\ref{l:6i} that
\begin{equation}
\label{e:yxc}
\psi\,\,\text{is $\FF\otimes\BE_\YS$-measurable}.
\end{equation}
Moreover, for every $\omega\in\Omega$, since
$\varphi_\omega\in\Gamma_0(\XS)$, $\varphi_\omega^*$ is proper
and hence $\epi\psi_\omega\neq\emp$.
On the other hand, the Fenchel--Moreau biconjugation theorem
yields
\begin{equation}
\label{e:8cp}
(\forall\omega\in\Omega)\quad
\vartheta(\omega)
=-\varphi_\omega^{**}\big(x(\omega)\big)
=-\varphi_\omega\big(x(\omega)\big)
\end{equation}
and it thus follows from \eqref{e:fxs} that $\vartheta$ is
$\FF$-measurable. Now define
\begin{equation}
\label{e:m3}
(\forall n\in\NN)\quad
M_n\colon\Omega\to 2^\YS\colon\omega\mapsto
\begin{cases}
\menge{\mathsf{y}\in\YS}{\psi(\omega,\mathsf{y})\leq-n},
&\text{if}\,\,\vartheta(\omega)=\minf;\\
\menge{\mathsf{y}\in\YS}{\psi(\omega,\mathsf{y})
\leq\vartheta(\omega)+2^{-n}},
&\text{if}\,\,\vartheta(\omega)\in\RR.
\end{cases}
\end{equation}
Fix temporarily $n\in\NN$. By \eqref{e:yxc},
$\menge{(\omega,\mathsf{y})}{\mathsf{y}\in
M_n(\omega)}\in\FF\otimes\BE_\YS$. Hence, since
$(\YS,\EuScript{T}_\YS)$ is a Souslin space,
\cite[Theorem~5.7]{Himm75} guarantees that there exist
$y_n\in\mathcal{L}(\Omega;\YS)$
and $B_n\in\FF$ such that $\mu(B_n)=0$ and
$(\forall\omega\in\complement B_n)$ $y_n(\omega)\in M_n(\omega)$.
Now set $B=\bigcup_{n\in\NN}B_n$. Then $\mu(B)=0$ and, by virtue
of \eqref{e:r3} and \eqref{e:m3},
\begin{equation}
\big(\forall\omega\in\complement B\big)(\forall n\in\NN)\quad
\vartheta(\omega)
\leq\inf_{k\in\NN}\psi\big(\omega,y_k(\omega)\big)
\leq\psi\big(\omega,y_n(\omega)\big)
\leq
\begin{cases}
-n,&\text{if}\,\,\vartheta(\omega)=\minf;\\
\vartheta(\omega)+2^{-n},&\text{if}\,\,\vartheta(\omega)\in\RR.
\end{cases}
\end{equation}
Thus, letting $n\uparrow\pinf$ yields
$(\forall\omega\in\complement B)$
$\vartheta(\omega)=\inf_{n\in\NN}\psi(\omega,y_n(\omega))$.
Consequently, since $\YY$ is compliant, property~\ref{t:1ii} in
Theorem~\ref{t:1} is satisfied. In turn, we deduce from
\eqref{e:8cp}, Theorem~\ref{t:1}, \eqref{e:63}, and
Theorem~\ref{t:2}\ref{t:2i} that
\begin{align}
\mathfrak{I}_{\varphi,\TXX}(\widetilde{x})
&=\int_\Omega\varphi\big(\omega,x(\omega)\big)\mu(d\omega)
\nonumber\\
&=-\int_\Omega\inf_{\mathsf{y}\in\YS}\psi(\omega,\mathsf{y})\,
\mu(d\omega)
\nonumber\\
&=-\inf_{y\in\YY}\int_\Omega\psi\big(\omega,y(\omega)\big)
\mu(d\omega)
\nonumber\\
&=\sup_{y\in\YY}\bigg(
\int_\Omega\Pair{x(\omega)}{y(\omega)}_{\XS,\YS}\mu(d\omega)
-\int_\Omega\varphi_\omega^*\big(y(\omega)\big)\mu(d\omega)\bigg)
\nonumber\\
&=\sup_{\widetilde{y}\in\widetilde{\YY}}
\big(\pair{\widetilde{x}}{\widetilde{y}}-
\mathfrak{I}_{\varphi,\TXX}^*(\widetilde{y})\big)
\nonumber\\
&=\mathfrak{I}_{\varphi,\TXX}^{**}(\widetilde{x}).
\end{align}
Thus $\mathfrak{I}_{\varphi,\TXX}=
\mathfrak{I}_{\varphi,\TXX}^{**}$ and,
since $\mathfrak{I}_{\varphi,\TXX}$ is proper, we conclude
that $\mathfrak{I}_{\varphi,\TXX}\in\Gamma_0(\TXX)$.

\ref{p:8ii}:
The normality of $\varphi$ implies that it is
$\FF\otimes\BE_\XS$-measurable and that
there exists $u\in\mathcal{L}(\Omega;\XS)$ such that
$(\forall\omega\in\Omega)$ $u(\omega)\in\dom\varphi_\omega$.
Hence, for every $n\in\NN$, the function
$(\Omega\times\XS,\FF\otimes\BE_\XS)\to\RX\colon
(\omega,\mathsf{x})\mapsto
\varphi_\omega(u(\omega)+n\mathsf{x})-\varphi_\omega(u(\omega))$
is measurable. Since, by \eqref{e:r},
\begin{equation}
(\forall\omega\in\Omega)(\forall\mathsf{x}\in\XS)\quad
(\rec\varphi)(\omega,\mathsf{x})
=(\rec\varphi_\omega)(\mathsf{x})
=\lim_{\NN\ni n\uparrow\pinf}
\dfrac{\varphi_\omega\big(u(\omega)+n\mathsf{x}\big)
-\varphi_\omega\big(u(\omega)\big)}{n},
\end{equation}
it follows that $\rec\varphi$ is $\FF\otimes\BE_\XS$-measurable.
Now let $\widetilde{x}\in\TXX$ and
$\widetilde{z}\in\dom\mathfrak{I}_{\varphi,\TXX}$. Then, for
$\mu$-almost every $\omega\in\Omega$,
$z(\omega)\in\dom\varphi_\omega$ and it thus follows from the
convexity of $\varphi_\omega$ that the function
$\theta\colon\RPP\to\RX\colon\alpha\mapsto
(\varphi_\omega(z(\omega)+\alpha x(\omega))
-\varphi_\omega(z(\omega)))/\alpha$ is increasing.
Thus, appealing to \eqref{e:r} and the monotone convergence
theorem, we deduce from \ref{p:8i} that
\begin{align}
\big(\rec\mathfrak{I}_{\varphi,\TXX}\big)(\widetilde{x})
&=\lim_{\alpha\uparrow\pinf}
\frac{\mathfrak{I}_{\varphi,\TXX}(\widetilde{z}+
\alpha\widetilde{x})-\mathfrak{I}_{\varphi,\TXX}(\widetilde{z})}{
\alpha}
\nonumber\\
&=\lim_{\alpha\uparrow\pinf}\int_\Omega\frac{
\varphi_\omega\big(z(\omega)+\alpha x(\omega)\big)
-\varphi_\omega\big(z(\omega)\big)}{\alpha}\mu(d\omega)
\nonumber\\
&=\int_\Omega\lim_{\alpha\uparrow\pinf}\frac{
\varphi_\omega\big(z(\omega)+\alpha x(\omega)\big)
-\varphi_\omega\big(z(\omega)\big)}{\alpha}\,\mu(d\omega)
\nonumber\\
&=\int_\Omega(\rec\varphi_\omega)\big(x(\omega)\big)\mu(d\omega),
\end{align}
as claimed.
\end{proof}

Two key ingredients in Hilbertian convex analysis are the Moreau
envelope of \eqref{e:7} and the proximity operator of \eqref{e:8}
\cite{Livre1,More65}. To compute them for integral functions, we
first observe that, in the case of Hilbert spaces identified with
their duals, Assumption~\ref{a:2} can be simplified as follows.

\begin{assumption}\
\label{a:3}
\begin{enumerate}[label={\rm[\Alph*]}]
\item
\label{a:3a}
$\XS$ is a separable real Hilbert space with scalar product
$\scal{\Cdot}{\Cdot}_\XS$, associated norm $\|\Cdot\|_\XS$,
and strong topology $\EuScript{T}_\XS$.
\item
\label{a:3b}
$(\Omega,\FF,\mu)$ is a $\sigma$-finite measure space such that
$\mu(\Omega)\neq 0$.
\item
\label{a:3c}
$\XX=\menge{x\in\mathcal{L}(\Omega;\XS)}{
\int_\Omega\|x(\omega)\|_\XS^2\,\mu(d\omega)<\pinf}$ and
$\TXX$ is the usual real Hilbert space
$L^2(\Omega;\XS)$ with scalar product
\begin{equation}
(\forall\widetilde{x}\in\TXX)(\forall\widetilde{y}\in\TXX)\quad
\scal{\widetilde{x}}{\widetilde{y}}=
\int_\Omega\scal{x(\omega)}{y(\omega)}_\XS\,\mu(d\omega).
\end{equation}
\item
\label{a:3e}
$\varphi\colon(\Omega\times\XS,\FF\otimes\BE_\XS)\to\RX$ is a
normal integrand such that $(\forall\omega\in\Omega)$
$\varphi_\omega\in\Gamma_0(\XS)$.
\item
\label{a:3f}
$\dom\mathfrak{I}_{\varphi,\TXX}\neq\emp$ and
$\dom\mathfrak{I}_{\varphi^*,\TXX}\neq\emp$.
\end{enumerate}
\end{assumption}

\begin{proposition}
\label{p:11}
Suppose that Assumption~\ref{a:3} holds and
let $\gamma\in\RPP$. Then the following are satisfied:
\begin{enumerate}
\item
\label{p:11i}
Let $\widetilde{x}\in\TXX$ and $\widetilde{p}\in\TXX$. Then
$\widetilde{p}=\prox_{\gamma
\mathfrak{I}_{\varphi,\TXX}}
\widetilde{x}$ $\Leftrightarrow$
$p(\omega)=\prox_{\gamma\varphi_\omega}(x(\omega))$
for $\mu$-almost every $\omega\in\Omega$.
\item
\label{p:11ii}
Set $\moyo{\varphi}{\gamma}\colon\Omega\times\XS\to\RX\colon
(\omega,\mathsf{x})\mapsto\moyo{(\varphi_\omega)}{\gamma}
(\mathsf{x})$.
Then $\moyo{\varphi}{\gamma}$ is normal and
$\moyo{\mathfrak{I}_{\varphi,\TXX}}{\gamma}
=\mathfrak{I}_{\moyo{\varphi}{\gamma},\TXX}$.
\end{enumerate}
\end{proposition}
\begin{proof}
Since Assumption~\ref{a:3} is an instance of Assumption~\ref{a:2},
we first infer from Proposition~\ref{p:8}\ref{p:8i} that 
$\mathfrak{I}_{\varphi,\TXX}\in\Gamma_0(\TXX)$.

\ref{p:11i}: We derive from \eqref{e:8} and
Theorem~\ref{t:2}\ref{t:2ii} that
\begin{align}
\widetilde{p}=\prox_{\gamma\mathfrak{I}_{\varphi,
\TXX}}\widetilde{x}
&\Leftrightarrow\widetilde{x}-\widetilde{p}\in\gamma\partial
\mathfrak{I}_{\varphi,\TXX}(\widetilde{p})
\nonumber\\
&\Leftrightarrow x(\omega)-p(\omega)\in\gamma\partial
\varphi_\omega\big(p(\omega)\big)\,\,
\text{for $\mu$-almost every}\,\,\omega\in\Omega
\nonumber\\
&\Leftrightarrow p(\omega)=\prox_{\gamma\varphi_\omega}
x(\omega)\,\,\text{for $\mu$-almost every}
\,\,\omega\in\Omega.
\end{align}

\ref{p:11ii}:
Since $\BE_{\XS\times\RR}=\BE_\XS\otimes\BE_\RR$,
it results from Assumption~\ref{a:3}\ref{a:3e} and
Definition~\ref{d:n} that there exists a sequence
$(\boldsymbol{x}_n)_{n\in\NN}$
in $\mathcal{L}(\Omega;\XS\times\RR)$ such that
\begin{equation}
\label{e:usn}
(\forall\omega\in\Omega)\quad
\epi\varphi_\omega=
\overline{\big\{\boldsymbol{x}_n(\omega)\big\}_{n\in\NN}}.
\end{equation}
Set $\boldsymbol{\mathsf{V}}=
\menge{(\mathsf{x},\xi)\in\XS\times\RR}{
\|\mathsf{x}\|_\XS^2/(2\gamma)<\xi}$.
Then $\boldsymbol{\mathsf{V}}$ is open and therefore,
for every $\boldsymbol{\mathsf{C}}\subset\XS\times\RR$,
$\boldsymbol{\mathsf{C}}+\boldsymbol{\mathsf{V}}
=\overline{\boldsymbol{\mathsf{C}}}+\boldsymbol{\mathsf{V}}$.
Thus, we derive from
\eqref{e:7} and \eqref{e:usn} that
\begin{align}
(\forall\omega\in\Omega)\quad
\menge{(\mathsf{x},\xi)\in\XS\times\RR}{
\moyo{(\varphi_\omega)}{\gamma}(\mathsf{x})<\xi}
&=\menge{(\mathsf{x},\xi)\in\XS\times\RR}{
\varphi_\omega(\mathsf{x})<\xi}+\boldsymbol{\mathsf{V}}
\nonumber\\
&=\overline{\menge{(\mathsf{x},\xi)\in\XS\times\RR}{
\varphi_\omega(\mathsf{x})<\xi}}+\boldsymbol{\mathsf{V}}
\nonumber\\
&=\epi\varphi_\omega+\boldsymbol{\mathsf{V}}
\nonumber\\
&=\overline{\big\{\boldsymbol{x}_n(\omega)\big\}_{n\in\NN}}
+\boldsymbol{\mathsf{V}}
\nonumber\\
&=\big\{\boldsymbol{x}_n(\omega)\big\}_{n\in\NN}
+\boldsymbol{\mathsf{V}}
\nonumber\\
&=\bigcup_{n\in\NN}\big(\boldsymbol{x}_n(\omega)
+\boldsymbol{\mathsf{V}}\big).
\end{align}
Hence, for every $\mathsf{x}\in\XS$ and every $\xi\in\RR$,
since
$(\mathsf{x},\xi)-\boldsymbol{\mathsf{V}}\in\BE_{\XS\times\RR}$
and $\{\boldsymbol{x}_n\}_{n\in\NN}\subset
\mathcal{L}(\Omega;\XS\times\RR)$, we obtain
\begin{equation}
\menge{\omega\in\Omega}{
\moyo{(\varphi_\omega)}{\gamma}(\mathsf{x})<\xi}
=\Menge{\omega\in\Omega}{
(\mathsf{x},\xi)\in\bigcup_{n\in\NN}\big(\boldsymbol{x}_n(\omega)
+\boldsymbol{\mathsf{V}}\big)}
=\bigcup_{n\in\NN}\boldsymbol{x}_n^{-1}\big(
(\mathsf{x},\xi)-\boldsymbol{\mathsf{V}}\big)
\in\FF,
\end{equation}
which shows that $(\moyo{\varphi}{\gamma})(\Cdot,\mathsf{x})$ is
$\FF$-measurable.
Hence, since $(\XS,\EuScript{T}_\XS)$ is a Fr\'echet space,
Theorem~\ref{t:3}\ref{t:3iib} ensures that
$\moyo{\varphi}{\gamma}$ is normal. It remains to show that
$\moyo{\mathfrak{I}_{\varphi,\TXX}}{\gamma}
=\mathfrak{I}_{\moyo{\varphi}{\gamma},\TXX}$.
Let $\widetilde{x}\in\TXX$ and set
$\widetilde{p}=\prox_{\gamma
\mathfrak{I}_{\varphi,\TXX}}
\widetilde{x}$.
Then, by \ref{p:11i}, for 
$\mu$-almost every $\omega\in\Omega$,
$p(\omega)=\prox_{\gamma\varphi_\omega}(x(\omega))$
and, therefore, \eqref{e:7b} yields
$\moyo{(\varphi_\omega)}{\gamma}(x(\omega))
=\varphi_\omega(p(\omega))
+\|x(\omega)-p(\omega)\|_\XS^2/(2\gamma)$. Hence
\begin{align}
\moyo{\mathfrak{I}_{\varphi,\TXX}}{\gamma}(\widetilde{x})
&=\mathfrak{I}_{\varphi,\TXX}(\widetilde{p})
+\frac{1}{2\gamma}\|\widetilde{x}-\widetilde{p}\|_{\TXX}^2
\nonumber\\
&=\int_\Omega\varphi_\omega\big(p(\omega)\big)\mu(d\omega)
+\frac{1}{2\gamma}\int_\Omega
\|x(\omega)-p(\omega)\|_\XS^2\mu(d\omega)
\nonumber\\
&=\int_\Omega
\moyo{(\varphi_\omega)}{\gamma}\big(x(\omega)\big)\mu(d\omega)
\nonumber\\
&=\mathfrak{I}_{\moyo{\varphi}{\gamma},\TXX}(\widetilde{x}),
\end{align}
which concludes the proof.
\end{proof}

\begin{remark}\
\label{r:5}
Theorem~\ref{t:2}, Proposition~\ref{p:8}, and
Proposition~\ref{p:11} extend the state of the art on several
fronts, in particular by removing completeness of
$(\Omega,\FF,\mu)$ when $\XS$ is infinite-dimensional.
\begin{enumerate}
\item
The conclusion of Theorem~\ref{t:2}\ref{t:2i} first appeared in
\cite[Theorem~2]{Roc68a} in the special case when $\XS$ is the
standard Euclidean space $\RR^N$ and $\XX$ is
Rockafellar-decomposable (see Proposition~\ref{p:10}\ref{p:10iv}
for definition).
\item
In view of Proposition~\ref{p:10}\ref{p:10iv} and
Theorem~\ref{t:3}\ref{t:3ia}, Theorem~\ref{t:2} subsumes
\cite[Theorem~2 and Equation~(25)]{Rock71} (see also
\cite[Theorem~21]{Rock74}), where $\XS$ is a separable Banach
space, $\XX$ is Rockafellar-decomposable,
and $(\Omega,\FF,\mu)$ is complete.
\item
The conclusion of Theorem~\ref{t:2}\ref{t:2i} appears in
\cite{Vala75} in the special case when $\XX$ is
Valadier-decomposable (see Proposition~\ref{p:10}\ref{p:10v} for
definition) and $(\Omega,\FF,\mu)$ is complete.
\item
Proposition~\ref{p:8}\ref{p:8i} subsumes
\cite[Corollary p.~227]{Rock71}, where $\XS$ is a separable Banach
space, $\XX$ is Rockafellar-decomposable, and $(\Omega,\FF,\mu)$ is
complete.
\item
The conclusion of Proposition~\ref{p:8}\ref{p:8ii} first appeared
in \cite[Proposition~1]{Bism73} in the context where
$\XS$ is a separable reflexive Banach space,
$\XX$ is Rockafellar-decomposable,
and $(\Omega,\FF,\mu)$ is a complete probability space.
Another special case is \cite[Theorem~2]{Penn18},
where $\XX$ is Valadier-decomposable and either $\XS=\RR^N$ or
$(\Omega,\FF,\mu)$ is complete.
\item
Proposition~\ref{p:11}\ref{p:11i} appears in 
\cite[Proposition~24.13]{Livre1} in the special case when 
$(\Omega,\FF,\mu)$ is complete, for every $\omega\in\Omega$
$\varphi_\omega=\mathsf{f}$, and either $\mu(\Omega)<\pinf$ or
$\mathsf{f}\geq\mathsf{f}(\mathsf{0})\geq 0$.
\end{enumerate}
\end{remark}

\end{document}